\documentclass[12pt,a4paper, reqno]{amsart}
\usepackage{amssymb}
\usepackage{fontenc}
\usepackage{amsfonts}
\usepackage{amsxtra}
\usepackage{amsmath}
\usepackage{amssymb}
\usepackage{amscd}
\usepackage{amsthm}
\usepackage{mathrsfs}
\usepackage{url}
\usepackage[all]{xypic}
\usepackage{graphicx}
\usepackage{amsmath}
\usepackage{hyperref}
\usepackage{url}

\usepackage{multirow}
\hypersetup{linkcolor=blue, urlcolor=green, pdfstartview=FitH }
\makeatletter
\def\resetMathstrut@{%
  \setbox\z@\hbox{%
    \mathchardef\@tempa\mathcode`\(\relax
    \def\@tempb##1"##2##3{\the\textfont"##3\char"}%
    \expandafter\@tempb\meaning\@tempa \relax
  }%
  \ht\Mathstrutbox@1.2\ht\z@ \dp\Mathstrutbox@1.2\dp\z@
}
\usepackage{amsmath}
\hypersetup{colorlinks=true,citecolor=red,linkcolor=blue, urlcolor=green, pdfstartview=FitH }

\numberwithin{equation}{section}
\newtheorem*{theorem*1}{Theorem A}
\newtheorem*{proposition*2}{Proposition B}
\newtheorem*{proposition*1}{Proposition C}
\newtheorem{lemma}{Lemma}[section]
\newtheorem{proposition}[lemma]{Proposition}

\newtheorem{example}[lemma]{Example}
\newtheorem{theorem}[lemma]{Theorem}
\newtheorem{definition}[lemma]{Definition}

\newtheorem{corollary}[lemma]{Corollary}

\baselineskip 20pt \textwidth 18cm  \theoremstyle{plain}

\newcommand{\End}{\operatorname{End}}
\newcommand{\Hom}{\operatorname{Hom}}
\renewcommand{\Im}{\operatorname{Im}}
\newcommand{\Ker}{\operatorname{Ker}}

\newcommand{\Irr}{\operatorname{Irr}}

\newcommand{\Ind}{\operatorname{Ind}}
\newcommand{\Rep}{\operatorname{Rep}}
\newcommand{\Res}{\operatorname{Res}}

\newcommand{\Ha}{\operatorname{H}}

\newcommand{\Supp}{\operatorname{Supp}}
\newcommand{\HS}{\operatorname{HS}}

\newcommand{\C}{\mathbb C}

\newcommand{\R}{\mathbb R}

\newcommand{\GL}{\operatorname{GL}}
\newcommand{\GU}{\operatorname{GU}}
\newcommand{\GSp}{\operatorname{GSp}}

\newcommand{\Sp}{\operatorname{Sp}}

\newcommand{\T}{\mathbb{T}}
\newcommand{\GO}{\operatorname{GO}}

\newcommand{\diag}{\operatorname{diag}}

\newcommand{\id}{\operatorname{Id}}

\newcommand{\Tr}{\operatorname{Tr}}

\begin{document}
\title{Notes on unitary theta representations of compact groups}
\date{\today}
\author{Chun-Hui Wang }
\address{School of Mathematics and Statistics\\Wuhan University \\Wuhan, 430072,
P.R. CHINA}
\email{cwang2014@whu.edu.cn}
\begin{abstract}
We continue our work on understanding  Howe correspondences by using theta representations from $p$-adic groups to compact groups. We prove some results for unitary  theta representations of compact groups with respect to the induction and restriction  functors.
\end{abstract}
\maketitle
\tableofcontents
\section{Introduction}
This note is to continue our work on understanding  theta correspondence or Howe correspondence from the point of view of pure representation theory. Howe correspondences  originally arise  from  the Weil representation(\cite{We2})   by considering its restriction to  reductive dual pairs(cf.\cite{Ho1}\cite{Ku2}\cite{MVW}). According to many pioneering works (\cite{Ho1}\cite{Ho2}\cite{Ge}\cite{MVW}\cite{Ku}\cite{Ra}\cite{Wa},etc.), this kind of  correspondence   expands out in many different directions. Relevant articles and lectures by experts (\cite{Ga}\cite{Pr}\cite{Ho3},etc) also shed light on this topic.  Our main  purpose   is to work out some similar results  from  $p$-adic groups  \cite{Wan} to compact groups. In this article, we focus on  unitary representations of compact groups. In this certain part we mainly refer to the following books and papers: \cite{BeHa}\cite{Br}\cite{CH}\cite{Fa}\cite{Fo}\cite{HM1}\cite{HM2}\cite{HM}\cite{KT}.

  Let us begin by recalling the relevant  notations and definitions sketchily. Let $G$ be   a second-countable compact Hausdorff topological group. Let $\Rep(G)$(resp. $\widehat{G}$) denote the set of unitary equivalence classes of (resp. irreducible)  unitary representations of $G$.  Analogous of representations of $p$-adic groups, for $\pi \in \Rep(G) $, we let  $\mathcal{R}_{G}(\pi)=\{ \rho \in \widehat{G} \mid \Hom_G(\pi, \rho) \neq 0\}$.  Let us now consider the two-group case. Let  $(\Pi, V) \in \Rep(G_1\times G_2)$. For $(\pi_1, V_1)\in \mathcal{R}_{G_1}(\Pi)$, let $V_{\pi_1}$ denote the greatest $\pi_1$-isotypic quotient of $(\Pi,V)$.    By  Waldspurger's Lemmas on local radicals(cf. Lemmas \ref{waldspurger1}, \ref{waldspurger2}), $V_{\pi_1} \simeq \pi_1 \widehat{\otimes} \Theta_{\pi_1}$, for some  $ \Theta_{\pi_1}\in \Rep(G_2)$.  If for each $\pi_1\in \mathcal{R}_{G_1}(\Pi)$, $\mathcal{R}_{G_2}(\Theta_{\pi_1})$  contains only one element $\theta_{\pi_1} \in \widehat{G_2}$,  we can draw a map $\theta_1:  \mathcal{R}_{G_1}(\Pi) \longrightarrow \mathcal{R}_{G_2}(\Pi); \pi_1\longmapsto \theta_1(\pi_1)=\theta_{\pi_1}$;
in this case, we will call $\Pi$ a $\theta_1$-graphic or a graphic representation of $G_1\times G_2$. Similarly, we can define a $\theta_2$-graphic representation.
\begin{itemize}
\item[(1)] If $\Pi$ is a $\theta_1$ and $\theta_2$ representation, we will call $\Pi$ a $\theta$-bigraphic or a bigraphic representation of $G_1\times G_2$. In this case, $\pi_1 \longleftrightarrow \theta_1(\pi_1)$ defines a correspondence,  called theta correspondence or Howe correpondence,  between $\mathcal{R}_{G_1}(\Pi)$ and $\mathcal{R}_{G_2}(\Pi)$.
\item[(2)]  If $\Pi$ is a bigraphic and multiplicity-free representation of $G_1\times G_2$, we will call $\Pi$ a theta representation of $G_1\times G_2$.
\end{itemize}
These definitions originated  from the works of \cite{Ho1},\cite{Ho2},\cite{MVW}, etc. Now let us present the results that  we work out on these representations.  Let $H_1$, $H_2$ be two  normal closed  subgroups of $G_1$, $G_2$ respectively such that $\frac{G_1}{H_1} \simeq \frac{G_2}{H_2}$ under a map $\gamma$.  Let  $\Gamma(\gamma)$ denote its graph and assume  $\Gamma(\gamma) =\frac{\Gamma}{H_1\times H_2}$.   Let $(\rho, W)\in \Rep(\Gamma)$.
\begin{theorem*1}[Theorem \ref{Firstthm}]
$\Res_{H_1\times H_2}^{\Gamma} \rho$ is a theta representation of $H_1\times H_2$ iff $\Ind_{\Gamma}^{G_1\times G_2} \rho$  is a theta representation of $G_1\times G_2$.
\end{theorem*1}
Keep the assumptions. Assume now that $\frac{G_1}{H_1}$ is an abelian group. Without the consideration of the multiplicity, we proved the next two results:
  \begin{proposition*2}[Proposition \ref{Second}]
  For  $\rho \in \widehat{\Gamma}$, assume
 \begin{itemize}
\item[(1)]   $\Res_{H_1\times H_2}^{\Gamma} \rho$ is a bigraphic representation of $H_1\times H_2$,
\item[(2)] for each $\sigma\in \mathcal{R}_{H_1}(\rho)$, $\Ind_{H_1}^{G_1} \sigma$ is multiplicity-free.
\end{itemize}
Then  $\pi=\Ind_{\Gamma}^{G_1\times G_2} \rho$ is a $\theta_1$-graphic  representation of $G_1\times G_2$.
\end{proposition*2}
\begin{proposition*1}[Proposition \ref{thirdre}]
For  $\rho \in \widehat{\Gamma}$, assume
 \begin{itemize}
\item[(1)] $\pi=\Ind_{\Gamma}^{G_1\times G_2} \rho$  is a bigraphic representation of $G_1\times G_2$,
\item[(2)] for each $\pi_1\in \mathcal{R}_{G_1}(\pi)$, $ \Res_{H_1}^{G_1}\pi_1$ is multiplicity-free.
\end{itemize}
Then  $\Res_{H_1\times H_2}^{\Gamma} \rho$ is a $\theta_1$-graphic  representation of $H_1\times H_2$.
\end{proposition*1}
Known as a classical and manipulable theory for compact group, Mackey-Clifford theory is our main tool to verify the  above results.  As Rieffel pointed out in the introduction of \cite{Ri}, roughly speaking, Mackey-Clifford theory contains two parts for a normal group pair $(H, G)$ with $H\unlhd  G$. For $(\sigma, W)\in \widehat{H}$, let $I_G(\sigma) $ be the corresponding stability subgroup of $G$. The first part of  Mackey-Clifford theory studies the corresponding relationship between $\mathcal{R}_G(\Ind_H^G \sigma)$ and $\mathcal{R}_{I_G(\sigma)}(\Ind_H^{I_G(\sigma)}\sigma)$ by using the $\Ind$ and $\Res$ functors. The second part studies how to decompose irreducible elements of    $\mathcal{R}_{I_G(\sigma)}(\Ind_H^{I_G(\sigma)}\sigma)$ into projective representations of $I_G(\sigma)$ and $\frac{I_G(\sigma)}{H}$. To verify the results,  we  need to touch almost all  aspects of this two parts.  For such reason,  in section \ref{UNCG} we  collect some necessary results from the  references. The results include the following topics:  unitary  representations of compact groups(cf.\cite{Co}\cite{HM}\cite{KT}\cite{Ma4}),  Hilbert-Schmidt operator spaces(cf.,\cite{Bo}\cite{Ga2}), Induced representations(cf.\cite{KT}\cite{Ma1}), Frobenius reciprocity(cf.\cite{KT}\cite{Mo1}),  Mackey-Clifford theory(cf.\cite{CuRe}\cite{Is}\cite{KL}\cite{KT}\cite{Ma3}\cite{Ri}), Rieffel equivalence(cf.\cite{Ri}) and so on.  In section \ref{thetacompact}, we prove  our main results by using the tools coming from  Mackey-Clifford theory.  In the last section, we give some  easily considered examples;  for instance, we  give  some   theta representations of finite dimension from   Vign\'era's work \cite[Chap.1]{MVW} on reductive dual pairs.  It is our pleasure to thank Professor Guy Henniart  for insightful comments during the whole writing.  We also want to thank Professor Marc Rieffel for his suggestions. We  thank Professor Detlev Poguntke for   pointing  out some mistakes in the previous version.
\section{Unitary  representations of compact groups}\label{UNCG}

As is known that  unitary  representations of compact groups is a much classical field. Here we only collect some necessary results for us use in the next part.    Our main purpose of this section is to present  the Mackey-Clifford theory in the compact group case. We will mainly follow the notion and  conventions of  \cite{KT} and \cite{Ma4} to treat this part.

\subsection{Notation and conventions}\label{notation}
A compact  group means  a \emph{second-countable} compact \emph{Hausdorff} topological group in our circumstance  unless otherwise stated. Let $V, \langle, \rangle$ denote  a \emph{separable} Hilbert space, and $\mathcal{U}(V)$ the group of all unitary operators endowed with the strong operator topology. A unitary representation(cf.\cite[pp.346-347]{HM1},\cite[Lmm.1.1]{Ca2}) $\pi$ of $G$ is a continuous homomorphism $\pi: G \longrightarrow \mathcal{U}(V)$ with a $\C$-linear $G$-action on $V$; we also call $V$ a Hilbert  $G$-module or simply a $G$-module. A closed $G$-invariant   subspace of $V$  gives a unitary representation of $G$; we will call it  a  subrepresentation of $V$ or simply a  submodule of $V$. We call  $\pi$ irreducible if (1) $V\neq 0$, (2)  $0$ and $V$  are the only subrepresentations of $V$. To avoid confusion,  a usual  group representation  will be called  an algebraic  representation  or an algebraic module in this text.

For two unitary representations $(\pi_1, V_1)$, $(\pi_2, V_2)$ of $G$, let $\mathcal{B}(V_1, V_2)$ or $\Hom(V_1, V_2)$  denote the space of bounded linear operators from $V_1$ to $V_2$. Let $\Hom_G(\pi_1, \pi_2)$ or $\Hom_G(V_1, V_2)$ denote the set of all elements $T$ of $\mathcal{B}(V_1, V_2)$ satisfying $T\circ \pi_1(g)=\pi_2(g) \circ T $, for  all $g\in G$.   An element of $\Hom_G(\pi_1, \pi_2)$ is called a $G$-morphism or a $G$-intertwining operator from  $\pi_1$ to $\pi_2$.   If there exists a bijective unitary map $U: V_1 \longrightarrow V_2$ such that $U\circ \pi_1(g)=\pi_2(g)\circ U$,  we will say that $\pi_1$ is (unitarily) equivalent to $\pi_2$, and write $\pi_1\simeq \pi_2$. Let $\HS-\Hom_G(V_1, V_2)$ denote the subspace of $\Hom_G(V_1,V_2)$ of Hilbert-Schmidt operators.  Let $\widehat{G}$ denote the set of unitary equivalence classes of irreducible unitary representations of $G$, and  $\Rep(G)$ the set of equivalence classes of unitary representations of $G$.  For simplicity of notations, we often don't  distinguish a real  representation and its equivalence class in $\widehat{G}$.

For $(\pi, V)\in\Rep(G)$, let $(\overline{\pi}, \overline{V})$ denote the complex conjugate representation of $(\pi,V)$. For  $(\pi_1, V_1)$, $(\pi_2,V_2) \in \Rep(G)$,
  let $(\pi_1\widehat{\otimes} \pi_2, V_1\widehat{\otimes}V_2)$ denote the   Hilbert inner tensor product  representation of $G$ by  $(\pi_1, V_1)$ and $(\pi_2, V_2)$.
 \begin{itemize}
 \item[(1)] For $(\pi_1, V_1), (\pi_2, V_2)\in \Rep(G)$, denote $m_G(V_1, V_2)=\dim \Hom_G(V_1, V_2)$, which is a  natural number or infinite.
 \item[(2)] If $\pi_1\in \widehat{G}$, we will call $m_G(V_1, V_2)$ the \emph{multiplicity} of $\pi_1$ in $\pi_2$.
 \end{itemize}
Let us recall some notations from representations of $p$-adic groups(cf.\cite{BernZ},\cite{BushH},\cite{Ca}). For $\pi \in \Rep(G)$,  denote $\mathcal{R}_{G}(\pi)=\{ \rho \in \widehat{G} \mid \Hom_G(\pi, \rho) \neq 0\}$, and $\mathcal{L}_{G}(\pi)=\{ \rho \in \widehat{G} \mid \Hom_G(\rho, \pi) \neq 0\}$. In the compact case, $\mathcal{R}_{G}(\pi)=\Supp(\pi)=\mathcal{L}_{G}(\pi)$.

\subsubsection{Frobenius reciprocity}\label{Indrepre}
Let $H$ be a closed subgroup of $G$. For $(\sigma, W) \in \Rep(H)$, let $(\Ind_H^G \sigma, \Ind_H^G W)$ denote  the induced representation from $H$ to $G$ by $(\sigma, W)$.
By \cite[p.108, Thm.2.61]{KT} or \cite{Mo1}, we have:
\begin{theorem}[Frobenius reciprocity]\label{KT}
For $(\pi, V)\in \Rep(G)$, $(\rho, W) \in \Rep(H)$, if $\dim V<+\infty$, then $\Hom_G(\Ind_H^G \rho, \pi) \sim \Hom_H(\rho, \Res_H^G \pi)$,  equivalent as Banach spaces.
\end{theorem}
For later use, we can also obtain:
\begin{theorem}[Frobenius reciprocity]\label{HSSP}
For $(\pi, V) \in \Rep(G)$, $(\rho, W) \in  \Rep(H)$,
 $\HS-\Hom_G(\Ind_H^G \rho, \pi) \simeq \HS-\Hom_H(\rho, \Res_H^G \pi)$, as Hilbert spaces.
\end{theorem}
\begin{proof}
It suffices to verify that both sides are isomorphic as linear spaces.  Assume $\rho\simeq \widehat{\oplus}_{i=1}^{N_{\rho}}  \rho_i$, $ \pi\simeq \widehat{\oplus}_{j=1}^{N_{\pi}} \pi_j$, for some $\rho_i\in \widehat{H}$, $\pi_j\in \widehat{G}$, $N_{\rho}, N_{\pi}\in  \mathbb{N}^{\ast}=\mathbb{N}\cup \{+\infty\}$.  Hence $\HS-\Hom_G(\Ind_H^G \rho, \pi)\simeq \widehat{\oplus}_{i, j} \HS-\Hom_{G}(\Ind_H^G \rho_i, \pi_j)$, and $\HS-\Hom_H(\rho, \Res_H^G \pi)  \simeq \widehat{\oplus}_{i, j}\HS-\Hom_H(\rho_i, \Res_H^G \pi_j)$. Note that $\Hom_{G}(\Ind_H^G \rho_i, \pi_j) \simeq  \Hom_H(\rho_i, \Res_H^G \pi_j)$, and both spaces have the same finite  dimension.  Hence  $\HS-\Hom_{G}(\Ind_H^G \rho_i, \pi_j) \simeq \HS-\Hom_H(\rho_i, \Res_H^G \pi_j)$, and we are done!
\end{proof}
\subsubsection{Projective representations}
 Let us recall some basic results of projective representations of    compact groups  from \cite{Ma3}. Here we only collect some results for future  use, and one can see some
 related papers and books,  for examples  \cite{AM},\cite{Co},\cite{CH},\cite{HM1},\cite{KL},\cite{Mo0},\cite{Mo1},\cite{Mo2},\cite{Mo3},\cite{We} for details.
\begin{lemma}\label{mor}
 Let $G_1$, $G_2$ be two  separable\footnote{Here the separable means that $G_i$ contains a countable dense set. For the difference between separable and countably separated(\cite{Ma2}),  one can also refer to the  modern  book \cite{CH}.}   locally compact Hausdorff groups.  If $f:  G_1 \longrightarrow G_2$ is a  Borel group morphism, then it is also continuous.
 \end{lemma}
\begin{proof}
  Let $\mu_{G_1}$, $\mu_{G_2}$ be two left Haar measures on $G_1$, $G_2$ respectively.  Let $U_2$ be a non-empty  open neighborhood of the identity element    of $G_2$  with compact closure, so that  $0<\mu_{G_2}(U_2)<+\infty$.   Then $\exists$ an open subset $V_2$ of $G_2$, such that $1_{G_2}\in V_2$, $V_2^{-1}=V_2$, $V_2V_2\subseteq U_2$.  Let $\{x_n\}$ be a countable dense subset of $G_2$. For any $g\in G_2$, $\exists x_k$, such that $x_k\in gV_2$, so $g\in x_kV_2^{-1}=x_kV_2$. Hence $G_2=\cup_{n} x_n V_2$. Then $G_1=f^{-1}(\cup_{n} x_n V_2)=\cup_{n}f^{-1}( x_n V_2)$. Now each  $f^{-1}(  x_n V_2)$ is  a Borel set. Hence $0\neq \mu_{G_1}(G_1)\leq \sum_n \mu_{G_1}(f^{-1}( x_n V_2))$, which implies at least one $\mu_{G_1}(f^{-1}( x_n V_2))>0$.  By the general  regularized property(cf. Prop.7.2.6 in \cite{Co}) of $\mu_{G_1}$, there exists a compact subset $K_1 \subseteq  f^{-1}( x_n V_2)$, such that $0< \mu_{G_1}(K_1)< +\infty$. By \cite[p.296, Coro.]{HM1}, or  \cite[Appendice]{We}, $K_1^{-1}K_1$ contains a neighborhood of $1_{G_1}$. Hence $ f(K_1^{-1}K_1) = f(K_1)^{-1}f(K_1) \subseteq  V_2^{-1}x_n^{-1}x_n V_2=V_2^{-1}V_2 \subseteq U_2$, $f^{-1}(U_2)  \supseteq K_1^{-1}K_1$.  Hence $f$ is continuous.
\end{proof}
 Go back to the case that $G$ is  a second-countable compact Hausdorff group.  Let $\mathbb{T}=\{ x\in \C^{\times}\mid |x|=1\}$ be the cycle group.   Let $\Ha_m^2(G, \T)$ denote one  Moore measurable cohomology group as defined  in \cite{Mo1}. Here  we always choose  a  $2$-cycle $\alpha(-,-)$ such that $\alpha(1, g)=\alpha(g,1)=1$.\footnote{For the general cocycle, by the equality $\alpha(xy, z)\alpha(x,y)=\alpha(x,yz)\alpha(y,z)$, we can get $\alpha(g,1)=\alpha(1, g)=\alpha(1,1)$. To simply the discussion, one can  choose   a normalized cocycle as   given in  \cite{KL}. More precisely, assume now $\alpha(g,1)=1=\alpha(1,g)$. Then $1=\alpha(gg^{-1}, g)=\alpha(g^{-1}, g)/\alpha(g, g^{-1})$. Let us define $\varphi: G \longrightarrow \C^{\times}; g\longmapsto  \sqrt{\alpha(g, g^{-1})}$. Here, we choose a measurable map $\sqrt: \C^{\times} \longrightarrow \C^{\times}$. Consider  $\alpha'(g_1, g_2)= \alpha(g_1,g_2) \varphi(g_1)^{-1}\varphi(g_2)^{-1} \varphi(g_1g_2)$. Then $\alpha'(g, g^{-1})= 1$.} For such cocyle $\alpha$, we can associate to a  group $G^{\alpha}$ of elements $(g, t)$ with $g\in G$, $t\in \T$. The group law is defined by $(g_1, t_1) \cdot (g_2, t_2)=(g_1g_2, \alpha(g_1,g_2) t_1t_2)$.  Then $G^{\alpha}$ is a Borel group, under the product Borel structure of $G\times \T$. Moreover, $G\times \T$ is a Polish  space, and $G^{\alpha}$ is  a Polish group. Following   from the works of  Weil, Mackey,
we know that    $G^{\alpha}$ admits a uniquely locally compact group structure with the above Borel sets.
\begin{definition}\label{thedePro}
 A \emph{unitary $\alpha$-projective representation} $(\pi, V)$ of $G$ is a Borel  map $\pi: G \longrightarrow \mathcal{U}(V)$, for a separable Hilbert  vector space $V$, such that
    $\pi(g_1) \pi(g_2)= \alpha(g_1, g_2) \pi(g_1g_2)$,  for a  $2$-cocycle $\alpha(-,-)$ of some class in the measurable  cohomology $\Ha^2_m(G, \T)$.
  \end{definition}
  For each $\alpha$-projective representation $(\pi, V)$, we can associate a unitary representation $(\pi^{\alpha}, V^{\alpha})$ of $G^{\alpha}$, defined by $ \pi^{\alpha}([g,t]) v=t\pi(g)v$, for $g\in G$, $v\in V^{\alpha}=V$. Note that (1) the two canonical  morphisms $p_1: G^{\alpha} \longrightarrow G$, $\iota: \T  \longrightarrow G^{\alpha}$,  are continuous, and $p_2:   G^{\alpha}\longrightarrow \T $ is measurable, (2) the restriction $\pi^{\alpha}|_{\T}=id_{\T}$. Notice that by \cite[pp.346-347]{HM1}, $\pi^{\alpha}:  G^{\alpha} \longrightarrow \mathcal{U}(V)$ is strongly continuous.  Conversely, every unitary representation $(\rho^{\alpha}, W^{\alpha})$ satisfying $\rho^{\alpha}|_{\T}=id_{\T}$, arises from an $\alpha$-projective representation of $G$ in such form.

   Let $(\pi_i, V_i)$ be two non-zero $\alpha_i$-projective representations of $G$, $i=1,2$.  Call $T$ a $\phi$-projective $G$-intertwining operator or a  $\phi$-projective $G$-morphism,  from $\pi_1$ to  $\pi_2$ if (1)  $T: V_1 \longrightarrow V_2$, is a bounded linear operator, (2) $\phi: G \longrightarrow \C^{\times}$ is a Borel map, (3)  $T\circ \pi_1(g)=\phi(g)\pi_2(g)\circ T$, for any $g\in G$. If $T\neq 0$, $\phi(g)\in \T$.  Let $\Hom^{\phi}_G(V_1, V_2)$, or  $\Hom^{\phi}_G(\pi_1, \pi_2)$ denote all these operators.   If $\phi$ is the trivial map, we call $T$ a linear $G$-morphism.
\begin{lemma}
If $\Hom^{\phi}_G(V_1, V_2)\neq 0$, then
\begin{itemize}
\item[(1)] $\alpha_1$, $\alpha_2$ represent the same class in $\Ha_m^2(G, \T)$,
 \item[(2)] $\Hom^{\phi}_G(\pi_1, \pi_2)$ is a Banach subspace of $\mathcal{B}(V_1, V_2)$.
 \end{itemize}
\end{lemma}
\begin{proof}
1) Choose $v\in V_1$, such that $T(v)\neq 0$. For $g_1, g_2\in G$, $\phi(g_1)\phi(g_2) \alpha_2(g_1,g_2)\pi_2(g_1g_2)T(v)=\phi(g_1)\phi(g_2)\pi_2(g_1)\pi_2(g_2)T(v)=\phi(g_1)\pi_2(g_1)T(\pi_1(g_2)v)=T(\pi_1(g_1)\pi_1(g_2)v)=T(\pi_1(g_1g_2) \alpha_1(g_1, g_2)v)=\alpha_1(g_1, g_2)\phi(g_1g_2)\pi_2(g_1g_2)T(v)$, so $\phi(g_1)\phi(g_2)\phi^{-1}(g_1g_2) \alpha_2(g_1,g_2)=\alpha_1(g_1, g_2)$.\\
2) Let us   define a  right action of $G$ on $\mathcal{B}(V_1, V_2)$ by $ T^g=\phi^{-1}(g) \pi_2(g)^{-1}\circ T\circ \pi_1(g)$, for $g\in G$.  It can be checked that $T^{g_1g_2}=(T^{g_1})^{g_2}$, and $\phi(1)=1$. Notice that $T \longrightarrow T^g$, defines a continuous map on $\mathcal{B}(V_1, V_2)$. Then $\Hom^{\phi}_G(\pi_1, \pi_2)$ is the  $G$-stable subspace of $\mathcal{B}(V_1, V_2)$. Hence $\Hom^{\phi}_G(\pi_1, \pi_2)$ is a closed subspace of $\mathcal{B}(V_1, V_2)$.
\end{proof}
\begin{corollary}\label{IdPro}
If  $\Hom^{1}_G(V_1, V_2)\neq 0$, for  the trivial map $1$ from $G$ to $\T$,  then $\alpha_1=\alpha_2$.
\end{corollary}
\begin{corollary}
If $\alpha_1=\alpha_2$, then  $\Hom^{\phi}_G(V_1, V_2)= 0$ unless $\phi$ is a character from $G$ to $\T$. In particular, $\pi_1=\pi_2$, the class of all projective $G$-morphisms $\End^{\mathcal{P}}_G(V_1)=\cup_{\phi} \End^{\phi}_G(V_1)$, as $\phi$ runs through all characters of $G$.
\end{corollary}
\begin{corollary}\label{ISOP}
If $\alpha_1=\alpha_2$, $\Hom^{1}_G(V_1, V_2)\simeq \Hom_{G^{\alpha_1}}(V^{\alpha_1}_1, V^{\alpha_1}_2)$.
\end{corollary}
For the general case, assume $\alpha_1(g_1, g_2)=\phi(g_1)\phi(g_2)\phi^{-1}(g_1g_2) \alpha_2(g_1, g_2)$, for $g_i\in G$.  Then there exists a continuous  group isomorphism
$\iota_{\phi}: G^{\alpha_1} \longrightarrow G^{\alpha_2}; [g, t] \longrightarrow [g, \phi(g) t]$.  Let $\pi_2^{\alpha_2}\circ \iota_{\phi}$ be the  inflation representation of $ G^{\alpha_1}$ from $\pi_2^{\alpha_2}$ by $\iota_{\phi}$.  For such $\pi_i$, we have:
\begin{lemma}
$\Hom_G^{\phi}(\pi_1, \pi_2) \simeq \Hom_{G^{\alpha_1}}(\pi_1^{\alpha_1}, \pi_2^{\alpha_2} \circ \iota_{\phi})$, as Banach spaces.
\end{lemma}
\begin{proof}
The identity map $T \longrightarrow T$,  gives a bijection.
\end{proof}

 \begin{definition}
 \begin{itemize}
 \item[(1)]  Call $\pi_1$  projectively equivalent to $\pi_2$,  if there exists a bijective  unitary  $G$-morphism $T\in \Hom^{\phi}_G(V_1, V_2)$, for some $\phi$.
 \item[(2)] Call $\pi_1$  linearly  equivalent to $\pi_2$,  if there exists a bijective  unitary  $G$-morphism $T\in \Hom^{1}_G(V_1, V_2)$.
 \end{itemize}
 \end{definition}
 \begin{lemma}\label{isom}
Let $(\pi_i, V_i)$  be  irreducible $\alpha_i$-projective representations of $G$ of finite dimension, $i=1, 2$.  If $\Hom_G^1(\pi_1\otimes \pi_2, \C)\neq 0$, then $\pi_1$ is linearly isomorphic to $\overline{\pi_2}$.
\end{lemma}
\begin{proof}
By Lmm.\ref{IdPro}, $\alpha_2=\overline{\alpha_1}$. Then $\Hom_G^1(\pi_1\otimes \pi_2, \C)\simeq \Hom_G^1(\pi_1, \overline{\pi_2})\simeq \Hom_{G^{\alpha_1}}(\pi^{\alpha_1}_1, \overline{\pi_2}^{\alpha_1})$. The result then follows from Schur's Lemma.
\end{proof}
  Let $H$ be a closed subgroup of $G$. Let $(\sigma, W)$ be an $\omega$-projective representation of $G$, for some $2$-cocyle $\omega$ representing some class in $\Ha_m^2(H, \T)$.  Let $\Omega$ be a  $2$-cocyle representing    some class in $\Ha_m^2(G, \T)$, such that $\Omega|_{H\times H}=\omega$. Let $(\sigma^{\omega}, W^{\omega})$ be the ordinary  representation of $H^{\omega}$ arising  from   $(\sigma, W)$. We can define the induced representation  $(\Ind_{H^{\omega}}^{G^{\Omega}} \sigma^{\omega}, \Ind_{H^{\omega}}^{G^{\Omega}} W^{\omega})$. The restriction of it to  $G$ defines an $\Omega$-projective representation of $G$, called the projective induced  representation from $(\sigma^{\omega}, W^{\omega})$, denoted by $(\pi=\Ind_{H, \omega}^{G, \Omega} \sigma, V=\Ind_{H, \omega}^{G,\Omega} W)$.
\subsubsection{Mackey-Clifford's theory} Mackey-Clifford's theory is a vast subject. Here we shall  only  present the related results for later use by   following  \cite[\S11C]{CuRe}, \cite[Chap.5]{Is}, \cite{KT}, \cite{Ma3}, \cite{Ri}.  Let $H$ be a normal closed  subgroup of $G$.
\begin{theorem}[Clifford-Mackey]\label{cliffordadmissible1}
Let $(\pi, V) \in \widehat{G}$.  Then:
\begin{itemize}
\item[(1)] $\mathcal{R}_H(\pi) $ is a non-empty   finite set.
\item[(2)] If $\sigma_1, \sigma_2 \in \mathcal{R}_H(\pi)$, then there is an element $g\in G$ such that $\sigma_2 \simeq \sigma_1^g$, where $\sigma_1^g (h): =\sigma_1(ghg^{-1})$ for $h\in H$.
\item[(3)] There is a positive integer $m$ such that $\Res_H^G\pi\simeq \sum_{\sigma\in \mathcal{R}_H(\pi)} m \sigma$.
\item[(4)] Let $(\sigma, W) $ be an irreducible  constituent  of $(\Res_H^G\pi, \Res_H^G V)$. Then:
 \begin{itemize}
\item[(a)] $I_G(\sigma)=\left\{ g\in G \mid \sigma^g \simeq \sigma\right\}$ is an open subgroup of $G$.
\item[(b)] The isotypic component $m\sigma$ of $\sigma$ in $\Res_H^G \pi$ is an irreducible  representation of $I_G(\sigma)$, denoted by $(\widetilde{\sigma}, \widetilde{W})$.
\item[(c)] $\Res_{H}^G\pi= \oplus_{\sigma \in \mathcal{R}_{H}(\pi)} \widetilde{\sigma}$ with $\widetilde{\sigma}|_H \simeq m \sigma$.
\end{itemize}
\item[(5)]  The action of   $G$ on the set $\mathcal{R}_{H}(\pi)$ is transitive. Moreover, $\#\mathcal{R}_H(\pi) =[G: I_G(\sigma)]$.
\item[(6)] $\pi\simeq \Ind_{I_{G}(\sigma)}^G \widetilde{\sigma}$, for any $\sigma \in \mathcal{R}_{H}(\pi)$.

\end{itemize}
\end{theorem}
By  \cite[p.298, Thm.8.2]{Ma3}, $(\sigma, W)$ can extend to be an $\alpha$-projective representation of $I_{G}(\sigma)$, for some $2$-cocyle $\alpha(-,-)$ from $\frac{I_{G}(\sigma)}{H} \times \frac{I_{G}(\sigma)}{H} $ to $\T$.
Let $\mathcal{P}Irr(\frac{I_{G}(\sigma)}{H})$ (resp. $\mathcal{P}Irr_{\alpha^{-1}}(\frac{I_{G}(\sigma)}{H})$)  denote the set of \emph{linear} equivalence classes of irreducible unitary projective (resp. $\alpha^{-1}$-projective) representations of $\frac{I_{G}(\sigma)}{H}$.
\begin{theorem}[Clifford-Mackey]\label{cliffordadmissible2}
\begin{itemize}
\item[(1)] If $\widetilde{\sigma}_1 \in \mathcal{R}_{I_{G}(\sigma)}(\Ind_H^{I_{G}(\sigma)} \sigma)$, then
\begin{itemize}
\item[(a)] $\widetilde{\sigma}_1|_{H} \simeq m_1 \sigma$, for some positive integer $m_1$,
\item[(b)] $\pi_1=\Ind_{I_{G}(\sigma)}^G\widetilde{\sigma}_1$, which  is an irreducible representation of $G$,
\item[(c)] $\pi_1 \in \mathcal{R}_G(\Ind_H^G \sigma)$,
\item[(d)] $m_G (\pi_1, \Ind_H^G\sigma) =m_H(\widetilde{\sigma}_1, \sigma)=m_1$.
\end{itemize}
\item[(2)] $\Ind_{I_{G}(\sigma)}^G: \mathcal{R}_{I_{G}(\sigma)}(\Ind_H^{I_{G}(\sigma)} \sigma) \longrightarrow \mathcal{R}_G(\Ind_H^G \sigma);  \widetilde{\sigma}_i \longmapsto \Ind_{I_{G}(\sigma)}^G   \widetilde{\sigma}_i$, is a bijective map.
 \item[(3)] For any $\widetilde{\sigma}_i \in \mathcal{R}_{I_{G}(\sigma)}(\Ind_H^{I_{G}(\sigma)} \sigma)$, there exists a $\delta_i \in \mathcal{P}Irr_{\alpha^{-1}}(\frac{I_{G}(\sigma)}{H})$, such that $\sigma \otimes \delta_i \simeq \widetilde{\sigma}_i $, as ordinary representations.
\item[(4)]  $\sigma\otimes:  \mathcal{P}Irr_{\alpha^{-1}}(\frac{I_{G}(\sigma)}{H}) \longrightarrow \mathcal{R}_{I_{G}(\sigma)}(\Ind_H^{I_{G}(\sigma)} \sigma); \delta_i \longmapsto \sigma\otimes \delta_i$, is a well-defined bijective map.
\end{itemize}
\end{theorem}
Analogue of \cite[p.106, Thm.2.58]{KT}, we have:
\begin{lemma}\label{decom}
$\Ind_H^{I_G(\sigma)}\sigma \simeq \sigma\widehat{\otimes} \Ind_{H, \alpha^{-1}}^{I_G(\sigma),  \alpha^{-1}} \C$.
\end{lemma}
\begin{proof}
Note that $\Ind_{I_G(\sigma)^{\alpha}}^{I_G(\sigma)^{\alpha}} W^{\alpha} \widehat{\otimes} \Ind_{H^{ \alpha^{-1}}}^{I_G(\sigma)^{\alpha^{-1}}} \C^{\alpha^{-1}} \simeq \Ind_{I_G(\sigma)^{\alpha} \times H^{ \alpha^{-1}}}^{I_G(\sigma)^{\alpha} \times I_G(\sigma)^{\alpha^{-1}}} (W^{\alpha} \otimes \C^{\alpha^{-1}})$. Let $[I_G(\sigma) \times I_G(\sigma)]^{\alpha \times \alpha^{-1}}$ be the central group extension of $[I_G(\sigma) \times I_G(\sigma)]$ by $\T$ associated to $\alpha \times \alpha^{-1}$. Then there exists a short sequence $1\longrightarrow  \T \stackrel{\iota}{\longrightarrow}I_G(\sigma)^{\alpha} \times I_G(\sigma)^{\alpha^{-1}} \stackrel{\kappa}{\longrightarrow}  [ I_G(\sigma) \times  I_G(\sigma) ]^{\alpha \times \alpha^{-1}}  \longrightarrow 1$, where $\iota(t)=[(1, t), (1, t^{-1})]$, $\kappa([(g_1, t_1), (g_2, t_2)])=[(g_1, g_2), t_1t_2]$, for $t, t_i \in \T$, $g_i \in I_G(\sigma)$. Similarly, there exists another short sequence  $1\longrightarrow  \T \stackrel{\iota}{\longrightarrow}I_G(\sigma)^{\alpha} \times H^{\alpha^{-1}} \stackrel{\kappa}{\longrightarrow}  [ I_G(\sigma) \times H ]^{\alpha \times \alpha^{-1}} \longrightarrow 1$.  Let $\kappa': \frac{I_G(\sigma)^{\alpha} \times I_G(\sigma)^{\alpha^{-1}} }{\T} \longrightarrow  [I_G(\sigma) \times I_G(\sigma)]^{\alpha \times \alpha^{-1}}$ be the group isomorphism. By \cite[Prop.2.38]{KT}, $\Ind_{I_G(\sigma)^{\alpha} \times H^{ \alpha^{-1}}}^{I_G(\sigma)^{\alpha} \times I_G(\sigma)^{\alpha^{-1}}} (\sigma^{\alpha} \otimes \C^{\alpha^{-1}}) \simeq [\Ind_{[I_G(\sigma)\times H]^{\alpha \times  \alpha^{-1}}}^{[I_G(\sigma)\times I_G(\sigma)]^{\alpha \times\alpha^{-1}}} (\sigma \otimes \C)^{\alpha \times\alpha^{-1}}]\circ \kappa'  $. The group $I_G(\sigma)$ can embed in $[I_G(\sigma)\times I_G(\sigma)]^{\alpha \times\alpha^{-1}}$, with the image, denoted by $\Delta I_G(\sigma)$. Moreover $ \Delta I_G(\sigma)  [I_G(\sigma)\times H]^{\alpha \times  \alpha^{-1}} =[I_G(\sigma)\times I_G(\sigma)]^{\alpha \times\alpha^{-1}}$, and  $\Delta I_G(\sigma) \cap  [I_G(\sigma)\times H]^{\alpha \times  \alpha^{-1}} =\Delta H$.  Hence  $\Ind_{\Delta H}^{\Delta I_G(\sigma)}\sigma \otimes \C \simeq \Res_{\Delta I_G(\sigma)}^{{[I_G(\sigma)\times H]^{\alpha \times  \alpha^{-1}}}}\big(\Ind_{[I_G(\sigma)\times H]^{\alpha \times  \alpha^{-1}}}^{[I_G(\sigma)\times I_G(\sigma)]^{\alpha \times\alpha^{-1}}} (\sigma \otimes \C)^{\alpha \times\alpha^{-1}}\big)$, which is essentially the result.
\end{proof}
Let $p: I_G(\sigma) \longrightarrow \frac{I_G(\sigma)}{H}$ be the projection. Analogue of {\cite[Prop.2.38]{KT}}, we have:
\begin{lemma}\label{iso22}
$\Ind_{H, \alpha^{-1}}^{I_G(\sigma),  \alpha^{-1}} \C\simeq  (\Ind_{\frac{H}{H}, 1}^{\frac{I_G(\sigma)}{H},  \alpha^{-1}} \C)\circ p$, linear equivalence.
\end{lemma}
\begin{proof}
We can lift both projective representations to ordinary representations of $I_G(\sigma)^{\alpha^{-1}}$. The left hand side is the representation $\Ind_{H^{\alpha^{-1}}}^{I_G(\sigma)^{\alpha^{-1}}}\C^{\alpha^{-1}}$, and the right hand side is the representation $[\Ind_{(\frac{H}{H})^{ \alpha^{-1}}}^{(\frac{I_G(\sigma)}{H})^{\alpha^{-1}}} \C^{\alpha^{-1}}] \circ P$, where $P$ is the projection $ I_G(\sigma)^{\alpha^{-1}} \longrightarrow \frac{I_G(\sigma)^{\alpha^{-1}}}{H^{\alpha^{-1}}}$. By {\cite[Prop.2.38]{KT}}, these two ordinary representations are unitary equivalence. Hence the restricted projective representations are  linearly unitary equivalence.
\end{proof}

\subsection{Rieffel Equivalence}\label{RE}
 Following the idea of  Rieffel in \cite{Ri}, we can define some equivalent relations on   $\widehat{G}$ as well as $\widehat{H}$.
\begin{lemma}\label{eqiuv}
\begin{itemize}
\item[(1)] For $\pi_1, \pi_2\in \widehat{G}$, if $\mathcal{R}_{H}(\pi_1) \cap \mathcal{R}_{H}(\pi_2) \neq \emptyset  $ then  $\mathcal{R}_{H}(\pi_1) = \mathcal{R}_{H}(\pi_2)$.
    \item[(2)] For $\sigma_1, \sigma_2\in \widehat{H}$, if $\mathcal{R}_{G}(\Ind_{H}^G \sigma_1) \cap \mathcal{R}_{G}(\Ind_{H}^G \sigma_2) \neq \emptyset  $ then $\mathcal{R}_{G}(\Ind_{H}^G \sigma_1) =\mathcal{R}_{G}(\Ind_{H}^G \sigma_2) $.
\end{itemize}
\end{lemma}
\begin{proof}
1) Assume $\sigma \in \mathcal{R}_{H}(\pi_1) \cap \mathcal{R}_{H}(\pi_2) $. Then $\mathcal{R}_{H}(\pi_1)=\{ \sigma^g \mid g\in G\}=\mathcal{R}_{H}(\pi_2)$.\\
2) Assume $\pi \in \mathcal{R}_{G}(\Ind_{H}^G \sigma_1) \cap \mathcal{R}_{G}(\Ind_{H}^G \sigma_2)$. Then $\sigma_1, \sigma_2\in \mathcal{R}_H(\pi)$, and $\sigma_2\simeq \sigma_1^g$, for some $g\in G$.  By \cite[p.90, Pro.2.39]{KT}, $\Ind_H^G \sigma_1 \simeq \Ind_H^G \sigma_1^g \simeq \Ind_H^G \sigma_2$.  Hence the result holds.
\end{proof}
For $\pi_1, \pi_2 \in \widehat{G}$, we call  $\pi_1\sim_{(G, H)} \pi_2$ if $\mathcal{R}_{H}(\pi_1)= \mathcal{R}_{H}(\pi_2)$. Clearly, $\sim_{(G, H)}$ defines an equivalent relation on $\widehat{G}$. For $\sigma_1, \sigma_2 \in \widehat{H} $, we call $\sigma_1\sim_{(G, H)} \sigma_2$ if $\mathcal{R}_{G}(\Ind_H^G\sigma_1)= \mathcal{R}_{G}(\Ind_H^G\sigma_2)$.
\begin{lemma}\label{GHEQ}
\begin{itemize}
\item[(1)] If  $\pi_1\sim_{(G, H)} \pi_2$, then for any $\sigma_i \in \mathcal{R}_{H}(\Res_{H}^G \pi_i)$,  $\sigma_1\sim_{(G, H)} \sigma_2$.
\item[(2)] If  $\sigma_1\sim_{(G, H)} \sigma_2$, then for any $\pi_i \in \mathcal{R}_{G}(\Ind_{H}^G \sigma_i)$, $\pi_1\sim_{(G, H)} \pi_2$.
\end{itemize}
\end{lemma}
\begin{proof}
1) In this case, $\sigma_2 \simeq \sigma_1^g$, for some $g\in  G$,  so $\sigma_1 \sim_{(G, H)} \sigma_2$.\\
2) In this case, $\sigma_i \in \mathcal{R}_{H}(\pi)$, for some $\pi\in \widehat{G}$. Hence $\sigma_2\simeq \sigma_1^g$, for some $g\in G$. Then $\sigma_1^g \in \mathcal{R}_H(\pi_1) \cap \mathcal{R}_H(\pi_2)$. By the above lemma \ref{eqiuv}, $\pi_1\sim_{(G, H)} \pi_2$.
\end{proof}
Therefore without regard for the multiplicity, there exists the following correspondence:
\[ \frac{\widehat{H}}{\sim_{(G, H)}}\Longleftrightarrow^{\Ind_H^G}_{\Res_H^G}  \frac{\widehat{G}}{\sim_{(G, H)}}.\]
\begin{lemma}
If $G$ is an abelian group, then:
\begin{itemize}
\item[(1)] $ \frac{\widehat{H}}{\sim_{(G, H)}}$, $\frac{\widehat{G}}{\sim_{(G, H)}}$ both are also groups,
 \item[(2)] $\Ind_H^G:  \frac{\widehat{H}}{\sim_{(G, H)}} \longrightarrow \frac{\widehat{G}}{\sim_{(G, H)}}$ is a group isomorphism.
 \end{itemize}
\end{lemma}
\begin{proof}
 For  $\chi_1, \chi_2\in \widehat{H}$,   $\chi_1  \sim_{(G, H)} \chi_2$ iff $\exists g\in G$ such that $\chi_2\simeq \chi_1^g =\chi_1$ iff $\chi_1=\chi_2$.  For $\Sigma_1, \Sigma_2\in \widehat{G}$, then $\Sigma_1\sim_{(G, H)} \Sigma_2$ iff $\Sigma_2|_{H} =\Sigma_1|_{H}$.  So $ \frac{\widehat{H}}{\sim_{(G, H)}}$, $\frac{\widehat{G}}{\sim_{(G, H)}}$  are two isomorphic  groups.
\end{proof}
\subsection{Example} Assume that $\frac{G}{H}$ is an abelian compact group. Keep the notations. Note that $\frac{G}{H}$  is  compact, so the image of a character is also compact and lies in $\T$.
\begin{lemma}
If $\# \frac{G}{H}$ is a finite cyclic group, then $\Ha^2(\frac{G}{H}, \T)=0$.
\end{lemma}
\begin{proof}
See  \cite[Prop. 11.46, Coro. 11.47]{CuRe}.
\end{proof}
\begin{lemma}\label{abl}
Keep the notations of Theorem \ref{cliffordadmissible2}. If $\# \frac{G}{H}$ is a finite cyclic group, then $\Res_{H}^G \pi$ is multiplicity-free.
\end{lemma}
\begin{proof}
In this case, $[\alpha]\in \Ha^2(\frac{I_{G}(\sigma)}{H}, \T)=0$. So $(\frac{I_G(\sigma)}{H})^{\alpha^{-1}} \simeq \frac{I_G(\sigma)}{H}\times \T$, which is also an abelian group. Since $\widetilde{\sigma}_1 \simeq \sigma \otimes \delta_1$, as projective representations, and $\delta_1^{\alpha^{-1}}$ is an irreducible representation of $(\frac{I_G(\sigma)}{H})^{\alpha^{-1}}$,  we can  assert $\dim \delta_1=1$, and $\widetilde{\sigma}_1|_{H} \simeq \sigma$.
\end{proof}
\begin{corollary}\label{mutifree}
If $\# \frac{G}{H}$ is a finite abelian group of order $n$, and $n$ has no  square factor, then $\Res_{H}^G \pi$ is multiplicity-free.
\end{corollary}
 \begin{lemma}\label{CHI}
If $(\pi_1,V_1)$, $( \pi_2,V_2)\in \mathcal{R}_{G}(\Ind_H^G \sigma)$,  then $\pi_1\simeq \pi_2\otimes \chi$, for some $\chi\in \Irr(\frac{G}{H})=\widehat{(\frac{G}{H})}$.
\end{lemma}
\begin{proof}
   Consider the finite-dimensional vector space $\Hom_H(V_1, V_2)$.  Define an action of $G$ or $\frac{G}{H}$ on it by $\varphi^g=\pi_2 (g^{-1})\circ \varphi \circ \pi_1(g)$, for $g\in G$. Then there exists a non-zero element $F$, and $\chi \in  \Irr(\frac{G}{H})$  such that $F^g=\chi(g) F$, for any $g\in G$. Then $F\in \Hom_{G}(\pi_1, \pi_2\otimes \chi)$.  Since $\pi_1$, $\pi_2\otimes \chi$ both are irreducible representations, $\pi_1\simeq \pi_2\otimes \chi$.
\end{proof}

\begin{lemma}\label{CHI2}
Keep the notations of Thm.\ref{cliffordadmissible2}.  If $\delta_1 \in  \mathcal{P}Irr_{\alpha^{-1}}(\frac{I_{G}(\sigma)}{H})$, then $ \mathcal{P}Irr_{\alpha^{-1}}(\frac{I_{G}(\sigma)}{H})=\{ \delta_1\otimes \chi_i \mid \chi_i \in \widehat{(\frac{I_{G}(\sigma)}{H})}\}$. Moreover, $\#\{\chi_i\mid  \delta_1\otimes \chi_i\simeq \delta_1\}=(\dim \delta_1)^2$.
\end{lemma}
\begin{proof}
By Lmm.\ref{decom}, $\Ind_{H}^{I_G(\sigma)} \sigma \simeq \sigma\widehat{\otimes}\Ind_{H, \alpha^{-1}}^{I_G(\sigma),  \alpha^{-1}} \C$.  Let $\delta_1^{\alpha^{-1}}$ be the lifting representation of $I_G(\sigma)^{\alpha^{-1}}$ from $\delta_1$. Note that $H^{\alpha^{-1}}$ is also a normal subgroup of $I_G(\sigma)^{\alpha^{-1}}$, and $\frac{I_G(\sigma)^{\alpha^{-1}}}{H^{\alpha^{-1}}}\simeq \frac{I_G(\sigma)}{H}$. The inclusion  $ \C^{\alpha^{-1}} \hookrightarrow \delta_1^{\alpha^{-1}}$ implies $\Ind_{H^{\alpha^{-1}}}^{I_G(\sigma)^{\alpha^{-1}}} \C^{\alpha^{-1}} \hookrightarrow \Ind_{H^{\alpha^{-1}}}^{I_G(\sigma)^{\alpha^{-1}}}\delta_1^{\alpha^{-1}}$. The right side
$\Ind_{H^{\alpha^{-1}}}^{I_G(\sigma)^{\alpha^{-1}}}\delta_1^{\alpha^{-1}} \simeq \delta_1^{\alpha^{-1}} \widehat{\otimes} \Ind_{H^{\alpha^{-1}}}^{I_G(\sigma)^{\alpha^{-1}}} \C \simeq \delta_1^{\alpha^{-1}} \widehat{\otimes} \Ind_{\frac{H^{\alpha^{-1}}}{H^{\alpha^{-1}}}}^{\frac{I_G(\sigma)^{\alpha^{-1}}}{H^{\alpha^{-1}}}} \C\simeq\delta_1^{\alpha^{-1}} \widehat{\otimes}\Ind_{1}^{\frac{I_G(\sigma)}{H}} \C \simeq \delta_1^{\alpha^{-1}} \widehat{\otimes}L^2(\frac{I_G(\sigma)}{H})  $. Now $\frac{I_G(\sigma)}{H}$ is an abelian  group,  so the first statement follows. Let $\widetilde{\delta_1} =\sigma\otimes \delta_1$. Then $ m_{I_G(\sigma)}(\Ind_{H}^{I_G(\sigma)} \sigma, \widetilde{\delta_1})=m_H(\sigma, \widetilde{\delta_1})=\dim \delta_1$, $m_{I_G(\sigma)^{\alpha^{-1}}} (\Ind_{H^{\alpha^{-1}}}^{I_G(\sigma)^{\alpha^{-1}}} \C^{\alpha^{-1}}, \delta_1^{\alpha^{-1}})= m_{H^{\alpha^{-1}}} ( \C^{\alpha^{-1}}, \delta_1^{\alpha^{-1}})=\dim \delta_1$, $m_{I_G(\sigma)^{\alpha^{-1}}} (\Ind_{H^{\alpha^{-1}}}^{I_G(\sigma)^{\alpha^{-1}}} \delta_1^{\alpha^{-1}}, \delta_1^{\alpha^{-1}})= m_{H^{\alpha^{-1}}} ( \delta_1^{\alpha^{-1}}, \delta_1^{\alpha^{-1}})=(\dim \delta_1)^2$.  Hence the second statement holds.
\end{proof}
\begin{lemma}\label{abeq1}
Keep the notations of Thm.\ref{cliffordadmissible2}.  For $g\in I_{G}(\sigma) $, $\delta_1^g \simeq \delta_1$, and $\sigma^g $  is linearly equivalent to $ \sigma \otimes \chi_g$ as $I_{G}(\sigma)$-modules, for some $\chi_g\in \Irr(\frac{I_{G}(\sigma)}{H})$.
\end{lemma}
\begin{proof}
1) Since $\delta_1$ is an $\alpha^{-1}$-projective representation of $\frac{I_{G}(\sigma)}{H}$, for $t\in I_{G}(\sigma)$, $\delta_1^g (t)=\delta_1(gtg^{-1})=\delta_1(t)$.\\
2) Assume that the $2$-cocycle $\alpha$ satisfies $\alpha(t, t^{-1})=1$, for any $t\in I_{G}(\sigma)$. Consequently, $\sigma(t^{-1})=\sigma(t)^{-1}$. For any $g\in I_{G}(\sigma)$, let us define $\chi_g: I_{G}(\sigma)\longrightarrow \C^{\times}; t\longmapsto \alpha(t, g)/\alpha(g, t)$. Then for $t, s\in  I_{G}(\sigma)$,
$$\chi_g(ts)=\frac{\alpha(ts,g)}{\alpha(g,ts)}=\frac{ \alpha(t,sg)\alpha(s,g)/\alpha(t,s)}{ \alpha(g,t)\alpha(gt,s)/\alpha(t,s)}=\frac{ \alpha(t,sg)\alpha(s,g)}{ \alpha(g,t)\alpha(gt,s)}$$
$$=\frac{\alpha(s,g)}{ \alpha(g,t)}\frac{ \alpha(t,gs)}{\alpha(tg,s)}=\frac{\alpha(s,g)}{ \alpha(g,t)}\frac{ \alpha(t,g)}{\alpha(g,s)}=\chi_{g}(s)\chi_{g}(t).$$
Hence $\chi_g\in \Irr(\frac{I_{G}(\sigma)}{H})$. Moreover, let $F=\sigma(g): W \longrightarrow W$. Then $F(\sigma(t)w)=\sigma(g)\sigma(t)\sigma(g^{-1})\sigma(g)w=\sigma(g)\sigma(t)\sigma(g^{-1})F(w)$. Note that
$$\sigma(g)\sigma(t)\sigma(g^{-1})=\sigma(gtg^{-1}) \alpha(g,t) \alpha(gt, g^{-1})=\sigma(gtg^{-1}) \alpha(g,t) \alpha(tg, g^{-1})$$
$$=\sigma(gtg^{-1}) \alpha(g,t) /\alpha(t, g)=\sigma(gtg^{-1})\chi^{-1}_g(t)=\sigma^g(t)\chi^{-1}_g(t).$$
\end{proof}
\begin{lemma}\label{abeq2}
Keep the above notations. For $g\in I_G(\sigma)$, $\widetilde{\sigma}_1^g \simeq \widetilde{\sigma}_1$ iff $\delta_1\otimes \chi_g$  is linearly equivalent to $ \delta_1$ as $I_{G}(\sigma)$-modules
\end{lemma}
\begin{proof}
 Note that $\widetilde{\sigma}_1^g=\sigma^g\otimes \delta_1^g \simeq (\sigma\otimes \chi_g)\otimes \delta_1 \simeq \widetilde{\sigma}_1 \simeq \sigma\otimes \delta_1 $.
 By Thm.\ref{cliffordadmissible2}(4), $\chi_g\otimes \delta_1 \simeq  \delta_1$, linear isomorphism. By  Thm.\ref{cliffordadmissible2}(4), the converse also holds.
\end{proof}

For $\pi \in \widehat{G}$, assume   $\sigma\in \mathcal{R}_{H}(\pi)$, and $\pi\simeq \Ind_{I_{G}(\sigma)}^G \widetilde{\sigma}$, $\widetilde{\sigma}\simeq \sigma\otimes \delta$ as projective representations. Let $f=[G: I_{G}(\sigma)]$, $e=\dim(\delta)$.
\begin{lemma}\label{eqeq}
If $ \pi \simeq \pi\otimes \chi$, for any $\chi\in \widehat{(\frac{G}{H})}$, then:
\begin{itemize}
\item[(1)]  $[I_G(\sigma):H]=e^2$,   $[G:H]=e^2f$,
\item[(2)] $\Ind_H^{I_G(\sigma)} \sigma\simeq e \widetilde{\sigma}$.
\end{itemize}
\end{lemma}
\begin{proof}
  Under the hypothesis, $\Ind_{H}^G \pi \simeq \pi \widehat{\otimes} \Ind_{H}^G \C\simeq \pi\widehat{\otimes} \Ind_{1}^{\frac{G}{H}} \C\simeq \pi \widehat{\otimes} L^2(\frac{G}{H})$. Hence $\mathcal{R}_{G}(\Ind_H^G \pi)=\{\pi\otimes \chi\}=\{\pi\}$. So $\mathcal{R}_{G}(\Ind_{H}^G \sigma)=\{\pi\}$.  By Thm.\ref{cliffordadmissible2}, $\mathcal{R}_{I_{G}(\sigma)}(\Ind_H^{I_G(\sigma)} \sigma)=\{\widetilde{\sigma}\}$. Thus, $\#\frac{I_G(\sigma)}{H}=(\dim \delta)^2$ by the above lemma \ref{CHI2}. Hence $\# \frac{G}{H}=fe^2<+\infty$. Moreover, $m_H(\widetilde{\sigma}, \sigma)=m_{I_G(\sigma)}(\widetilde{\sigma}, \Ind_H^{I_G(\sigma)}\sigma)=e$.
\end{proof}
 Let $(\pi, V)\in \widehat{G}$.  For $(\sigma_1,W_1), ( \sigma_2, W_2)\in \mathcal{R}_{H}(\pi)$, $I_G(\sigma_1)=I_G(\sigma_2)$.
\begin{lemma}
The $2$-cocyles $\alpha_i(-,-)$ associated to $\sigma_1$, $\sigma_2$ respectively  as in Thm.\ref{cliffordadmissible2}, can be chosen to be  the same cocycle.
\end{lemma}
\begin{proof}
Assume $\sigma_2\simeq \sigma_1^{g}$, for some $g\in G$. For simplicity, assume they are equal. For $h_1, h_2\in I_G(\sigma_1)$, $\sigma_1^g(h_1)\sigma_1^g(h_2)=\sigma_1(gh_1g^{-1})\sigma_1(gh_2g^{-1})=\alpha_1(gh_1g^{-1}, gh_2g^{-1})\sigma_1^g(h_1h_2)$. Notice that $\alpha_1$ is defined from $\frac{I_{G}(\sigma_1)}{H} \times \frac{I_{G}(\sigma_1)}{H} $ to $\T$. In our case, $\frac{I_{G}(\sigma_1)}{H}$ and $ \frac{G}{H}$ both are abelian groups. Hence $\alpha_1(gh_1g^{-1}, gh_2g^{-1})=\alpha_1(h_1,h_2)$.
\end{proof}

%For the general compact  group $G$, one can   consider  the character group $\mathcal{X}_G$  of $G$.¡¡ We leave the reader to  continue the thinking by using the sheaf language.
%Let us consider the set of  its normal subgroups: $\{H_i\mid H_i\triangleright G \}$.   If $H_i \subseteq H_j$, $\pi_1\sim_{(H_j, G)}\pi_2$, then $\mathcal{R}_{H_j}(\pi_1)\cap \mathcal{R}_{H_j}(\pi_2)  \neq \emptyset$, so $\mathcal{R}_{H_i}(\pi_1)\cap \mathcal{R}_{H_i}(\pi_2)  \neq \emptyset$, i.e. $\pi_1\sim_{(H_i, G)}\pi_j$. So there exists a map $\iota_{j,i}: \frac{\widehat{G}}{\sim_{(H_j, G)}}  \longrightarrow \frac{\widehat{G}}{\sim_{(H_i, G)}}$.

\section{Theta representations of compact groups}\label{thetacompact}
   \subsection{$\rho$-isotypic component} For  a (unitary) representation $(\pi, V)$ of $G$, $\rho \in \widehat{G}$, we let $V_{\rho}$  denote the $\rho$-isotypic component of $V$. For  a countable family of Banach spaces $\{ B_i\}$, we use the following notions(cf. \cite[p.72]{Co1}):
 \begin{itemize}
 \item[(1)] $l^{\infty}(B_i)=\{ (\varphi_i) \in \prod_{i} B_i \mid \sup_i \|\varphi_i\| <+\infty\}$;
 \item[(2)]  $l^{p}(B_i)=\{ (\varphi_i)  \in \prod_{i} B_i \mid (\sum_i \|\varphi_i\|^p)^{\frac{1}{p}} <+\infty\}$.
\end{itemize}
 \begin{lemma}\label{ISOM}
  For $(\pi_i, V_i)\in \Rep(G)$,
\begin{itemize}
\item[(1)] $\Hom_G(V_1, V_2) \simeq l^{\infty}(B_{\rho})$, for   $B_{\rho}=\Hom_{G}(V_{1, \rho}, V_{2,\rho})$, $\rho\in \widehat{G}$,
\item[(2)]  $\HS-\Hom_G(V_1, V_2) \simeq l^{2}(H_{ \rho})$, for $H_{\rho}=\HS-\Hom_{G}(V_{1, \rho}, V_{2,\rho})$, $\rho\in \widehat{G}$.
\end{itemize}
\end{lemma}
\begin{proof}
1) Let $P_{\rho}$ be the projection from $V_2$ to $V_{2, \rho}$. Then for any $\varphi\in \Hom_G(V_1, V_2)$, $P_{\rho} \circ \varphi \in \Hom_G(V_1, V_{2,\rho})\simeq \Hom_G(V_{1, \rho}, V_{2,\rho})$. Moreover, $\|P_{\rho} \circ \varphi\|\leq \|\varphi\|$, so $(P_{\rho} \circ \varphi)_{\rho\in \widehat{G}} \in l^{\infty}(B_{\rho})$.
 Conversely, for any $(\phi_{\rho})\in l^{\infty}(B_{\rho})$, $v=\sum_{\rho\in \widehat{G}} v_{\rho}\in V_1$, we can define $\phi(v)=\sum_{\rho\in \widehat{G}}\phi_{\rho}(v_{\rho})$. Clearly, $ \|\phi(v)\|^2=\sum_{\rho\in \widehat{G}}\|\phi_{\rho}(v_{\rho})\|^2\leq (\sup_{\rho} \|\phi_{\rho}\|^2 )\|v\|^2=(\sup_{\rho} \|\phi_{\rho}\|)^2 \|v\|^2$. Hence $\phi \in \Hom_{G}(V_1, V_2)$. Moreover, the maps
 $$P: \Hom_G(V_1, V_2) \longrightarrow l^{\infty}(B_{\rho}); \varphi \longmapsto (P_{\rho} \circ \varphi),$$
 $$S: l^{\infty}(B_{\rho})  \longrightarrow \Hom_G(V_1, V_2); (\phi_{\rho})\longmapsto  \phi=\sum \phi_{\rho}$$
 are converse to each other, and $\|P\|\leq 1$, $\|Q\|\leq 1$, so   the two Banach spaces are isometric.\\
 2) Similar to the above proof.
\end{proof}
\subsection{Waldspurger's Lemmas on local radicals}
 The $p$-adic version of the next two results  came from \cite[pp.45-47]{MVW}.
\begin{lemma}\label{waldspurger1}
For $(\pi_1, V_1)\in \widehat{G_1}$, $(\pi_2, V_2)\in \Rep(G_2)$,  if a closed vector subspace  $W $  of $ V_1 \widehat{\otimes} V_2$ is  $G_1 \times G_2$-invariant, then there is a unique(up to unitarily equivalent) $G_2$-invariant closed subspace $V_2'$ of $V_2$ such that $W \simeq V_1 \widehat{\otimes} V_2'=V_1\otimes V_2'$.
\end{lemma}
\begin{proof}
Assume $W\neq 0$. Let $\{ e_1, \cdots, e_n\}$  be an orthonormal basis of $V_1$. Since $V_1$ has finite dimension,  $ V_1 \widehat{\otimes} V_2 =\oplus_{i=1}^n e_i \otimes V_2$. Let $V_2'=\{ v_2\in V_2\mid \exists v_1\neq 0, \textrm{ such that  } v_1\otimes v_2\in W\}$.  If $0\neq v_1\otimes v_2\in W$,  then $\pi_1(\mathbb{C}[G_1])v_1\otimes \pi_2(\mathbb{C}[G_2])v_2\subseteq W$, so $V_1\otimes v_2\subseteq W$. For $v_2', v_2 \in V_2'$, $c_1, c_2\in \mathbb{C}$,   $V_1 \otimes (c_1v_2'+c_2v_2) \subseteq V_1\otimes c_1v_2'+V_1\otimes c_2v_2\subseteq W$, so  $c_1v_2'+c_2v_2\in V_2'$.  Moreover, for $g_2\in G_2$, $V_1\otimes g_2v_2 \subseteq  W$. Hence $V_2'$ is a $G_2$-invariant vector subspace of $V_2$.  If  $0\neq v=\sum_{k=1}^m e_{i_k}\otimes w_k \in W$, for some $w_k\neq 0$, by Schur's Lemma, there exists $\epsilon_{j}\in \mathbb{C}[G_1]$ such that $\pi_1(\epsilon_{j}) e_{i_k}=\delta_{jk} e_{i_j}$, for $1\leq k\leq m$. Then $[\pi_1(\epsilon_{j}) \otimes \pi_2(1_{G_2})](v)=e_{i_j}\otimes w_j\in W$, which implies $w_j\in V_2'$. Hence $v\in V_1\otimes V_2'$, and $W=V_1\otimes V_2'=\oplus_{i=1}^n e_i\otimes V_2'$, which implies that $V_2'$ is a closed subspace of $V_2$.

If there exists another  $G_2$-invariant closed subspace $V_2''$ of $V_2$ such that $W \simeq V_1 \widehat{\otimes} V_2''$. Let $F$ be a unitary $G_1\times G_2$-isomorphism from $V_1\otimes V_2'$ to $V_1\otimes V_2''$. Let $\{ f_j\mid j\geq 1\}$ be an orthonormal basis of $V_2'$. By Schur's Lemma, $\exists $ a unique non-zero element $f_j'\in V_2''$, such that $F(v_1\otimes f_j)=v_1\otimes f_j'$, for any $v_1\in V_1$.  Hence $1=\|e_i \otimes f_j \|=\|F(e_i \otimes f_j) \|=\|e_i\otimes f_j'\|=\|f_j'\|$. For different $j,l$,  $0=\langle e_i\otimes f_j, e_i\otimes f_l\rangle= \langle F(e_i\otimes f_j), F( e_i\otimes f_l)\rangle=\langle e_i\otimes f'_j, e_i\otimes f'_l\rangle=\langle  f'_j, f'_l\rangle$.  If $\sum_{k}c_k f_k \neq 0$, then $F(e_i\otimes \sum_{k}c_k f_k)=e_i\otimes \sum_{k}c_k f'_k\neq 0$, which implies that those $ f'_k$ are linearly independent.  Hence $f_j \longrightarrow f_j'$, defines a unitary  linear map from $ V_2' $ to $ V_2''$, say $\varphi$. Then $F=1\otimes \varphi$, which implies that $\varphi$ is a $G_2$-isomorphism.
\end{proof}
\begin{lemma}\label{waldspurger2}
For  $(\pi_1, V_1) \in \widehat{G_1}$, $(\pi, V)\in \Rep(G_1 \times G_2)$. Suppose that $\cap \ker(\varphi)=0$ for all $\varphi\in \Hom_{G_1}(V, V_1)$. Then there is a unique(up to unitarily equivalent)  representation  $(\pi_2', V_2')$ of $G_2$ such that $\pi \simeq \pi_1 \widehat{\otimes} \pi_2'=\pi_1 \otimes \pi_2'$.
\end{lemma}
\begin{proof}
In this case, $V$ is a total $\pi_1$-isotypic representation.  We take the Hilbert space $V_2'=\HS-\Hom_{G_1}(V_1, V)$. Since $\dim V_1 <+\infty$, $\HS-\Hom_{G_1}(V_1, V)=\Hom_{G_1}(V_1,V)$. We can define a $G_2$-action on $V_2'$ by $g_2 \cdot \varphi(v_1)=\pi(g_2) \varphi(v_1)$, for $\varphi\in V_2'$, $g_2\in G_2$. Let $\{e_1, \cdots, e_n\}$ be an othogonormal basis of $V_1$. Then $\|g_2 \varphi\|_{HS}^2=\sum_{i=1}^n \|[g_2\varphi](e_i)\|^2=\| \varphi\|_{HS}^2$, i.e., $g_2$ is a unitary action on $V_2'$. Moreover, for $g_i\in G_2, \varphi_i \in V_2'$,
$$\| g_1\varphi_1-g_2\varphi_2\|_{HS}^2=\sum_{i=1}^n \|[g_1\varphi_1-g_2\varphi_2](e_i)\|^2$$
$$\leq \sum_{i=1}^n  (\|g_1\varphi_1(e_i) - g_1\varphi_2(e_i)\|+\|g_1\varphi_2(e_i) - g_2\varphi_2(e_i)\|)^2$$
$$\leq  \sum_{i=1}^n 2\|g_1\varphi_1(e_i) - g_1\varphi_2(e_i)\|^2 +\sum_{i=1}^n 2\|g_1\varphi_2(e_i) - g_2\varphi_2(e_i)\|^2$$
$$\leq 2\|\varphi_1-\varphi_2\|_{HS}^2+ \sum_{i=1}^n 2\|g_1\varphi_2(e_i) - g_2\varphi_2(e_i)\|^2.$$
Note that  for any $e_i$ and $\varphi_2$ fixed, $g\longrightarrow g\varphi_2(e_i)$,  is continuous on $G_2$. Hence for any $\epsilon>0$,  there exist open neighborhoods $\mathcal{U}(g_2)$, $\mathcal{U}(\varphi_2)$, such that for any $g_1\in\mathcal{U}(g_2)$, $\varphi_1\in \mathcal{U}(\varphi_2)$, we have
$$\|\varphi_1-\varphi_2\|_{HS}<\frac{\epsilon}{2}, \quad \|g_1\varphi_2(e_i) - g_2\varphi_2(e_i)\|<\frac{\epsilon}{2\sqrt{n}},$$
$$\| g_1\varphi_1-g_2\varphi_2\|_{HS}^2 < \epsilon^2.$$
Hence the $G_2$-action on $V_2'$ defines a unitary representation of $G_2$, denoted by $(\pi_2', V_2')$ from now on. For $v_1\in V_1$, $\varphi \in \Hom_{G_1}(V_1, V)$, we can define $\Phi(v_1, \varphi)=\varphi(v_1)\in V$. Then it is a $\C$-bilinear map. For $g_1\in G_1, g_2\in G_2$, $\Phi(g_1v_1\otimes g_2 \varphi)=g_2\varphi(g_1v)=[g_1, g_2]\varphi(v)$. If $v_1=\sum_{i=1}^n c_i e_i$, $\| \Phi(v_1\otimes  \varphi)\|^2=\|\varphi(v_1)\|^2=\|\sum_{i=1}^n c_i\varphi(e_i)\|^2 \leq (\sum_{i=1}^n |c_i|^2) \cdot (\sum_{i=1}^n \|\varphi(e_i)\|^2)=\|\sum_{i=1}^n c_ie_i\|^2 \cdot \|\varphi\|^2_{HS}=\|v_1\|^2\cdot \|\varphi\|^2_{HS}=\|v_1\otimes  \varphi\|^2$. Hence $\Phi\in \Hom_{G_1\times G_2}(V_1\otimes V_2', V)$. Let $W=\ker(\Phi)$, which is a closed $G_1\times G_2$-invariant subspace of $V_1\otimes V_2'$. By the above lemma, $W\simeq V_1\otimes V_2''$, for some closed subspace $V_2''$ of $V_2'$. For any $\varphi\in  V_2''$,  if $\varphi \neq 0$, then $0=\Phi(v_1\otimes \varphi)=\varphi(v_1)$, for all $v_1\in V_1$. So $\varphi=0$, a contradiction! Therefore $\Phi$ is injective.  For any $0\neq v\in V$,  $W_v=\pi(\C[G_1])v$ has finite dimension. Then $V_1\otimes \Hom_{G_1}(V_1, W_v)\simeq W_v$. Hence  $\Phi$ is also surjective, and  $\pi \simeq \pi_1 \widehat{\otimes} \pi_2' $, as $G_1\times G_2$-modules.  The uniqueness  follows similarly as above.
\end{proof}

\subsection{Some definitions} Let $(\pi, V) \in \Rep(G_1\times G_2)$.   By  Waldspurger's two lemmas,  for $(\pi_1, V_1)\in \mathcal{R}_{G_1}(\pi)$, $V_{\pi_1} \neq 0$, and $V_{\pi_1} \simeq \pi_1 \widehat{\otimes} \pi_2'$, for some  $ \pi_2'\in \Rep(G_2)$. From now on, we denote this $\pi_2'$ by $\Theta_{\pi_1}$.  Notice that $\mathcal{R}_{G_2}(\Theta_{\pi_1}) \subseteq \mathcal{R}_{G_2}(\pi)$. If for each $\pi_1\in \mathcal{R}_{G_1}(\pi)$, $\mathcal{R}_{G_2}(\Theta_{\pi_1})$ only contains one element $\theta_{\pi_1} \in \widehat{G_2}$,  we can draw a map $\theta_1:  \mathcal{R}_{G_1}(\pi) \longrightarrow \mathcal{R}_{G_2}(\pi); \pi_1\longmapsto \theta_1(\pi_1)=\theta_{\pi_1}$. To write smoothly, we give the following definitions. This is just for well writing.
\begin{definition}
Let $(\pi, V) \in \Rep(G_1\times G_2)$.
\begin{itemize}
\item[(1)] If for each $\pi_1\in \mathcal{R}_{G_1}(\pi)$, $\mathcal{R}_{G_2}(\Theta_{\pi_1})$ contains one element, we will call $\pi$ a $\theta_1$-graphic or a graphic representation of $G_1\times G_2$, associating to  the  map $\theta_1: \pi_1 \longrightarrow \theta_1(\pi_1)$. Similarly, we can define a $\theta_2$-graphic representation.
\item[(2)] If $\pi$ is a $\theta_1$ and $\theta_2$ representation, we call $\pi$ a $\theta$-bigraphic or bigraphic representation of $G_1\times G_2$. In this case, $\pi_1 \longleftrightarrow \theta(\pi_1)$ defines a correspondence,  called theta correspondence or Howe correpondence,  between $\mathcal{R}_{G_1}(\pi)$ and $\mathcal{R}_{G_2}(\pi)$,
\item[(3)]  If $\pi$ is a $\theta$-bigraphic and multiplicity-free representation of $G_1\times G_2$, we will $\pi$ a theta representation of $G_1\times G_2$.
\end{itemize}
\end{definition}
We remark that the similar definitions  can also be given for projective representations directly or by lifting onto their central extension groups.
\begin{example}\label{simple2}
Let $N$ be a closed normal subgroup of a compact group $G$. For any $g\in G$, $\C[Ng]$ is a theta representation of $N\times N$, where the action is defined as follows:  $[(n_1, n_2) \phi](ng)=\phi(n_1^{-1}ngn_2)$, for $n, n_1,n_2\in N$, $\phi \in \C[Ng]$.
\end{example}
\subsection{Some observations}
In this subsection, we will present some  observations from  the experts' works on Howe correspondences \cite{Ad1}\cite{Gan}\cite{Ho1}\cite{Ho2}\cite{MVW}\cite{Ra},etc. Keep the notations. Let us write $V=\oplus_{(\sigma,W)\in \widehat{G_1}} V_{\sigma}$, and $V_{\sigma}=\sigma\widehat{\otimes} \Theta_{\sigma}=\sigma\otimes \Theta_{\sigma}$.
\begin{lemma}\label{Homdi}
Let us write $B_{\sigma}=\mathcal{B}(\Theta_{\sigma}, \Theta_{\sigma})=\Hom(\Theta_{\sigma}, \Theta_{\sigma})$, $H_{\sigma}=\HS-\Hom(\Theta_{\sigma}, \Theta_{\sigma})$, $(\sigma,W)\in \widehat{G_1}$. Then:
\begin{itemize}
\item[(1)] $\Hom_{G_1}(V_{\sigma}, V_{\sigma}) \simeq B_{\sigma}$, and $ \HS-\Hom_{G_1}(V_{\sigma}, V_{\sigma})\simeq H_{\sigma}$,
\item[(2)] $\Hom_{G_1}(V, V)\simeq  l^{\infty}(B_{\sigma})$,
\item[(3)]$\HS-\Hom_{G_1}(V, V)  \simeq l^{2} (H_{\sigma})$.
\end{itemize}
\end{lemma}
\begin{proof}
(1)(a) For any $T\in \Hom(\Theta_{\sigma}, \Theta_{\sigma})$, $\id \otimes T \in \Hom_{G_1}(\sigma\otimes \Theta_{\sigma}, \sigma\otimes \Theta_{\sigma})$. Moreover, $\|\id \otimes T\|=\|\id\| \cdot\|T\|=\|T\|$. Conversely, assume $0\neq S\in \Hom_{G_1}(\sigma\otimes \Theta_{\sigma}, \sigma\otimes \Theta_{\sigma})$.  Let $e_1, \cdots, e_n$ (resp. $f_1, \cdots, f_n, \cdots$) be an orthonormal  basis of  $W$ (resp. $\Theta_{\sigma}$). Assume $0\neq S(e_i \otimes f_j)=\sum_{i=1}^n e_i \otimes f_i'$. Since $W$ is an irreducible $G_1$-module, there exists $\epsilon_{i}\in \C[G_1]$, such that $\sigma(\epsilon_i) e_j=\delta_{ij} e_i$, for any $1\leq j\leq n$. Hence $S(e_i\otimes f_j)=S(\sigma(\epsilon_i) e_i\otimes f_j) =\sum_{j=1}^n\sigma(\epsilon_i)e_j \otimes f_j'=e_i\otimes f_i'$; consequently let us write $f'_{j'}=f_i'$, and  then  $S(w\otimes f_j)=w\otimes f'_{j'}$, for any $w\in W$.  Here, $f'_{j'}$ is uniquely determined by $f_j$ and $S$. If $S(W\otimes f_j)=0$, we let $f_{j'}'=0$. By uniqueness, $f_j \longrightarrow f_{j'}'$, defines a linear map from $\Theta_{\sigma}$ to $\Theta_{\sigma}$; let us denote it by $T$. Hence $S=\id \otimes T$ on the whole space, and then $\|S\|=\|T\|$. Therefore, $T\in \Hom(\Theta_{\sigma}, \Theta_{\sigma})$, and the first statement holds. \\
(1)(b) If $T$ is a Hilbert-Schmidt operator, so is $S=\id\otimes T$. So the second statement also holds.\\
Parts (2)(3) follow from  Lmm.\ref{ISOM} and (1).
\end{proof}
Let  $\Delta_{G_i}=\{ (g,g) \mid g\in G_i\}$ be the diagonal subgroup of $G_i\times G_i$. For $(\pi, V) \in \Rep(G_1\times G_2)$,
 let $\mathcal{C}=\mathcal{B}(V, V)$, a $C^{\ast}$-algebra. Let  $A$ be the   subalgebra of $\mathcal{C}$ generated by $\pi([g_1, 1])$, for all $g_1\in G_1$  and $\mathcal{A}$ the  strong closure of $A$ in $\mathcal{C}$; let   $B$ be the   subalgebra  of $\mathcal{C}$ generated by  $\pi([1,g_2])$, for all $g_2\in G_2$ and $\mathcal{B}$ the  strong closure of $B$ in $\mathcal{C}$.   Let $A'$, $B'$ denote the the commutants of $A$, $B$ in $\mathcal{C}$ respectively.  By \cite[12.3]{Co2}, $\mathcal{A}=A''$, $\mathcal{B}=B''$.

    Let $\mathcal{M}_{i}=\HS-\Hom_{G_i}(V, V)$, for $1\leq i\leq 2$. For  $\phi\in \mathcal{M}_{i}$, $g, g'\in G_j$, $j\neq i$,  let  $(g', g)$ act on $\phi$ as follows: $[(g', g) \phi](v)=\pi(g')\phi(\pi(g^{-1})v)$, for $ v\in V$.
\begin{lemma}\label{Uni}
With the above actions, $ \mathcal{M}_{i}$ is a unitary $(G_j\times G_j)$-module, for $1\leq i\neq j\leq 2$.
\end{lemma}
\begin{proof}
Assume $i=1$, $j=2$. Let $\{e_1, \cdots, e_n, \cdots \}$ be an orthonormal basis of $V$.  For a fixed $0\neq \phi \in \mathcal{M}_{1}$,  there exists a  series of operators $0\neq \mu_m$  of  finite rank approaching $\phi$ under the $\HS$-norm.  For $(g', g)\in G_2\times G_2$, \[\|[(g', g) \phi\|_{HS}^2=\sum_{i\geq 1} \|\pi(g')\phi(\pi(g^{-1})e_i)\|^2=\sum_{i\geq 1} \|\phi(\pi(g^{-1})e_i)\|^2=\| \phi\|_{HS}^2\]

According to Lmm.\ref{Homdi}(3), $\mathcal{M}_{1}\simeq l^2(H_{\sigma})$. Moreover, each $\Theta_{\sigma}$ can be written as a Hilbert direct sum of its $G_2$-isotypic components. In this way, we may choose  those $\mu_m $, such that the vector space  $[\Ker(\mu_m)]^{\bot}$ and $\Im(\mu_m)$ only contain $(G_1\times G_2)$-finite elements. For any $\epsilon>0$, there exists one $m$, such that $\| \phi-\mu_m\|_{HS}^2< \frac{\epsilon^2}{64}$. For such $\mu_m$, let  $U_{\mu_m}$ denote the vector space  spanned  by   $(h_1,h_2) v$,   for $(h_1,h_2)\in G_1\times G_2$, $v\in [\Ker(\mu_m)]^{\bot}+\Im(\mu_m)$. Then $U_{\mu_m}$ has finite dimension.  Hence  for  a fixed $(g', g)\in G_2\times G_2$, there exists a neighborhood $\mathcal{U}(g', g)$ such that for any $(h', h)\in  \mathcal{U}(g', g)$,
$$\|\pi(h')-\pi(g')\|_{U_{\mu_m}}< \frac{\epsilon}{4\|\mu_m\|_{HS}},\quad   \quad \|\pi(h^{-1})-\pi(g^{-1})\|_{U_{\mu_m}}< \frac{\epsilon}{4\|\mu_m\|_{HS}}.$$
Consequently,
$$\| \mu_m\circ[\pi(h^{-1})-\pi(g^{-1})]\|_{HS}\leq \|\mu_m\|_{HS} \cdot \|\pi(h^{-1})-\pi(g^{-1})\|_{U_{\mu_m}}< \frac{\epsilon}{4}$$
 $$\|[\pi(h')- \pi(g')]\circ \mu_m \|_{HS}\leq \|\mu_m\|_{HS} \cdot \|\pi(h')- \pi(g')\|_{U_{\mu_m}}<\frac{\epsilon}{4}.$$
Hence:
$$\|(h', h) \phi-(g', g)\phi\|_{HS}^2=\sum_{i\geq 1} \|[(h', h) \phi-(g', g)\phi](e_i)\|^2= \sum_{i\geq1}  \|h'\phi(h^{-1}e_i) -g'\phi(g^{-1}e_i)\|^2 $$ $$ \leq  \sum_{i\geq 1} (\|h'\phi(h^{-1}e_i) - g'\phi(h^{-1}e_i)\|+\|g'\phi(h^{-1}e_i) - g'\phi(g^{-1}e_i)\|)^2 $$
$$\leq 2 \|\pi(h')\circ \phi - \pi(g')\circ\phi\|_{HS}^2+ 2 \|\phi\circ \pi(h^{-1}) - \phi\circ \pi(g^{-1})\|_{HS}^2$$
$$\leq 4 \|[\pi(h')- \pi(g')]\circ \mu_m \|_{HS}^2+ 4\|[\pi(h')- \pi(g')]\circ (\phi-\mu_m )\|_{HS}^2$$
$$+4 \| \mu_m\circ[\pi(h^{-1}) -  \pi(g^{-1})]\|_{HS}^2+4 \| (\phi-\mu_m )\circ[\pi(h^{-1}) -  \pi(g^{-1})]\|_{HS}^2$$
$$\leq 4 \|[\pi(h')- \pi(g')]\circ \mu_m \|_{HS}^2+32\|\phi-\mu_m \|_{HS}^2+ 4\| \mu_m\circ[\pi(h^{-1}) -  \pi(g^{-1})]\|_{HS}^2$$
$$< \epsilon^2.$$
Hence for a fixed $\phi$, $(h',h) \longrightarrow (h',h)\phi$, defines a  continuous map on $G_2\times G_2$. Hence  $ \mathcal{M}_{1}$  is a unitary $(G_2\times G_2)$-module.
\end{proof}
\begin{lemma}\label{conjiso}
Keep the notations. $\mathcal{M}_{i} \simeq \oplus_{\sigma\in \widehat{G_i}}\overline{\Theta_{\sigma}} \widehat{\otimes} \Theta_{\sigma}$, as $G_j\times G_j$-modules, for $1\leq i\neq j\leq 2$.
\end{lemma}
\begin{proof}
It follows from Lemmas \ref{Homdi}, \ref{Uni}, by considering the action of $G_j\times G_j$.
\end{proof}
 \begin{proposition}\label{theta}
The following statements are equivalent:
\begin{itemize}
\item[(1)] $(\pi, V)$ is a theta representation of $G_1\times G_2$,
\item[(2)] $\mathcal{B}=\mathcal{A}'$,
\item[(3)]  $\mathcal{A}=\mathcal{B}'$,
\item[(4)] $\mathcal{R}_{G_j \times G_j}(\mathcal{M}_i)=\{ \overline{\sigma_j}\otimes \sigma_j \mid \textrm{ some }  \sigma_j\in \widehat{G_j}\}$, for $1\leq i\neq j\leq 2$,
\item[(5)]  $\mathcal{M}_i$ is a multiplicity-free representation of  $G_j\times G_j$, for $1\leq i\neq j\leq 2$,
\item[(6)] $m_{ G_j \times G_j}( \mathcal{M}_i, \overline{\sigma_j}\otimes \sigma_j ) \leq 1 $, for all $\sigma_j \in  \widehat{G_j}$,  $1\leq i\neq j\leq 2$.
    \end{itemize}
\end{proposition}
\begin{proof}
$(1) \Rightarrow (2)$ As  $\mathcal{A}=A''$, $\mathcal{A}'=A'$.  Clearly, $B \subseteq A'$, so the strong closure $\mathcal{B} \subseteq A'=\mathcal{A}'$. Up to isomorphism, let us write $V=\oplus_{(\sigma,W_{\sigma})\in \mathcal{R}_{G_1}(V)} W_{\sigma}\otimes \Theta_{\sigma}$. Since $\pi$ is a theta representation, $\Theta_{\sigma}=\theta(\sigma)$, which corresponds to $\sigma$ uniquely. In this case, $A'=\Hom_{G_1}(V, V) \simeq l^{\infty}(\Hom(\Theta_{\sigma}, \Theta_{\sigma}))$. Here, $\Theta_{\sigma}$ is an irreducible representation of finite dimension, so   $\Hom(\Theta_{\sigma}, \Theta_{\sigma})=\End(\Theta_{\sigma})$. For any different  irreducible representations $\sigma_1, \cdots, \sigma_n$, it is known that $B$ contains $\oplus_{i=1}^n \id\otimes \End(\Theta_{\sigma_i})$ by Schur's Lemma. Hence the strong closure $\mathcal{B}=A'$.\\
$(2) \Leftrightarrow (3)$ $\mathcal{B} =\mathcal{A}'$ implies $\mathcal{B}' =\mathcal{A}''=A''=\mathcal{A}$. The converse also holds. \\
$(2) \Leftrightarrow (3) \Rightarrow (1)$ Let $P_{\sigma}$ be the projection from $V$ to $W_{\sigma} \otimes \Theta_{\sigma}$. Then $P_{\sigma} A'P_{\sigma} \simeq  \Hom (\Theta_{\sigma}, \Theta_{\sigma})$; $P_{\sigma} BP_{\sigma}$, which is isomorphic to the image of $\pi(\id\otimes \C[G_2])$ in $\id \otimes\Hom (\Theta_{\sigma}, \Theta_{\sigma})$. Assume that  $\Theta_{\sigma}$ is  decomposed as $\Theta_{\sigma} \simeq \mathcal{H}_1\oplus \mathcal{H}_2\oplus \cdots$ as $G_2$-modules. If $\Theta_{\sigma}$ is not irreducible, $P_{\sigma} \mathcal{B}P_{\sigma}$ can not be  the full set $\id \otimes\Hom (\Theta_{\sigma}, \Theta_{\sigma})$. So $\pi$ is multiplicity-free, and it is a $\theta_1$-representation. By symmetry, $\pi$ is a $\theta$-representation. \\
$(1) \Rightarrow (4)$ By Lmm.\ref{conjiso}, it is right.\\
$(4) \Rightarrow (5)$ By Lmm.\ref{conjiso}, for any $\sigma\in \widehat{G_i}$,  $\Theta_{\sigma}=0$, or $\Theta_{\sigma} $ is an irreducible representation of $G_j$.  By symmetry, we   assume $i=1$, $j=2$.  If there exists two different irreducible representations $\sigma_1$, $\sigma_2$ of $G_1$ such that $\Theta_{\sigma_1} \simeq \Theta_{\sigma_2}\neq 0$, written by $\delta$. Then $\Theta_{\delta}$ contains at least $\sigma_1\oplus \sigma_2$, a contradiction. Hence $\mathcal{M}_{1}$ is a multiplicity-free representation of  $G_{2}\times G_{2}$.\\
$(5) \Rightarrow (6)$ It is clearly right.\\
$(6) \Rightarrow (1)$ If there exists $\sigma \in \widehat{G_1}$, such that $\Theta_{\sigma}\neq 0$, and $\Theta_{\sigma}$ is not irreducible. In case  (a),  if $\Theta_{\sigma}$  contains two different irreducible components $\delta_1$, $\delta_2$, then $\Theta_{\delta_1}$, $\Theta_{\delta_2}$ both contain $\sigma$ as subrepresentations. Hence
$m_{G_1 \times G_1}( \mathcal{M}_{2}, \overline{\sigma}\otimes \sigma ) \geq 2$, a contradiction. In case (b), if $\Theta_{\sigma}$  contains an
 irreducible component $\delta$ with multiplicity bigger than $2$, then $m_{ G_2 \times G_2}( \mathcal{M}_1,  \overline{\delta}\otimes \delta ) \geq 2$, also a contradiction. By symmetry, $\pi$ is a multiplicity-free, $\theta_1$ and $\theta_2$ representation.
\end{proof}
 Recall the definition of a strong Gelfand pair from \cite{AAG}.
  \begin{lemma}
\begin{itemize}
\item[(1)] If  $\overline{\pi}\widehat{\otimes}\pi|_{\Delta_{G_1} \times (G_2\times G_2)}$, $\overline{\pi}\widehat{\otimes} \pi|_{(G_1\times G_1) \times \Delta_{G_2} }$ both are multiplicity-free representations, then $\pi$ is a theta representation of $G_1\times G_2$.
      \item[(2)] Assume that each  $(\Delta_{G_i}, G_i\times G_i)$ is  a strongly   Gelfand pair,  for   $i=1,2$.  Then  if $\pi$ is a theta representation, then  $\overline{\pi}\widehat{\otimes} \pi|_{\Delta_{G_1} \times (G_2\times G_2)}$, $\overline{\pi}\widehat{\otimes} \pi|_{(G_1\times G_1) \times \Delta_{G_2} }$ both are multiplicity-free.
\end{itemize}
 \end{lemma}
\begin{proof}
The first statement  follows from the above (6).  For the second statement,  $\overline{\pi}\widehat{\otimes} \pi\simeq \widehat{\oplus}_{\sigma, \sigma' \in \mathcal{R}_{G_1}(\pi) } \overline{\sigma} \otimes \sigma'\otimes \overline{\theta_{\sigma}}\otimes \theta_{\sigma'}$,  for $\theta_{\sigma}, \theta_{\sigma'}  \in \mathcal{R}_{G_2}(\pi)$. Under the assumption, $ [ \overline{\theta_{\sigma}}\otimes\theta_{\sigma'} ]|_{\Delta_{G_2}}$ is  multiplicity-free, so  $m_{(G_1\times G_1) \times \Delta_{G_2}}(\overline{\pi}\widehat{\otimes} \pi,  \overline{\sigma}\otimes \sigma' \otimes \eta)\leq 1$, for any $\eta\in \widehat{G_2}$. By symmetry, the second statement holds.
\end{proof}

\subsection{Main results}
Let $G_1$, $G_2$ be two compact groups with two normal subgroups $H_1$, $H_2$ respectively such that $\frac{G_1}{H_1} \simeq \frac{G_2}{H_2}$ under a map $\gamma$.  Then the graph $\Gamma(\gamma) \subseteq \frac{G_1}{H_1}\times  \frac{G_2}{H_2}$.  Let $p: G_1\times G_2 \longrightarrow \frac{G_1}{H_1}\times  \frac{G_2}{H_2}$ be the projection.  Let  $\Gamma $\footnote{If $\Gamma$ is replaced  by another closed subgroup $\Gamma'$ which also contains $H_1\times H_2$ as a normal subgroup, and $\frac{\Gamma'}{H_1\times H_2} \simeq \frac{\Gamma}{H_1\times H_2}$, then the main result also holds. Because we only needs to change the isomorphism $\gamma$ by another $\gamma'$ such that $\Gamma'$  is  the inverse image of $\Gamma(\gamma')$ in $G_1\times G_2$.}  be the inverse image of $\Gamma(\gamma)$  in $G_1\times G_2$ under the map $p$. Clearly, $\Gamma$ is a closed subgroup of $G_1\times G_2$. Let $(\rho, W)\in \Rep(\Gamma)$.
\begin{theorem}\label{Firstthm}
$\Res_{H_1\times H_2}^{\Gamma} \rho$ is a theta representation of $H_1\times H_2$ iff $\pi=\Ind_{\Gamma}^{G_1\times G_2} \rho$  is a theta representation of $G_1\times G_2$.
\end{theorem}
We divide the proof in the next two subsections.

\subsubsection{$(\Rightarrow)$}
 Assume $\pi_1\otimes \pi_2=\pi_1\widehat{\otimes}\pi_2 \in \mathcal{R}_{G_1\times G_2}(\pi) $. By Frobenius reciprocity(Thm.\ref{KT}), $0\neq \Hom_{G_1\times G_2}(\pi, \pi_1\otimes \pi_2) \simeq \Hom_{\Gamma}(\rho, \pi_1\otimes \pi_2)$. So we can find $\rho_{0}\in \mathcal{R}_{\Gamma}(\rho)\cap  \mathcal{R}_{\Gamma}( \pi_1\otimes \pi_2)$. Assume $\sigma\otimes \delta\in \mathcal{R}_{H_1\times H_2}(\rho_{0})$. Since $m_{H_1\times H_2}(\rho, \sigma\otimes \delta)=1$, $m_{H_1\times H_2}(\rho_{0}, \sigma\otimes \delta)=1=m_{\Gamma}(\rho, \rho_0)$. For any $g_1\in G_1$, assume $(g_1, \gamma(g_1))\in \Gamma$. Then  $\rho_{0}^{(g_1, \gamma(g_1))} \simeq \rho_{0}$ as $\Gamma$-modules, and also  as $H_1\times H_2$-modules. Hence $\sigma^{g_1} \simeq \sigma$ iff $\delta^{\gamma(g_1)} \simeq \delta$. So $\gamma: \frac{I_{G_1}(\sigma)}{H_1} \simeq \frac{I_{G_2}(\delta)}{H_2}$ with the graph $\Gamma(\gamma)_{(\sigma, \delta)}= [(I_{G_1}(\sigma) \times I_{G_2}(\delta)) \cap\Gamma]/{H_1\times H_2}$. Moreover, $ [(I_{G_1}(\sigma) \times I_{G_2}(\delta)) \cap\Gamma]=I_{\Gamma}(\sigma\otimes \delta)$. Assume $\pi_1\simeq \Ind^{G_1}_{I_{G_1}(\sigma)}\widetilde{\sigma}$, $\pi_2\simeq \Ind^{G_2}_{I_{G_2}(\delta)}\widetilde{\delta}$, $\rho_0\simeq \Ind^{\Gamma}_{I_{\Gamma}(\sigma\otimes \delta)}\widetilde{\sigma\otimes \delta}$. Assume $\widetilde{\sigma}\simeq \sigma\otimes \sigma_1$, $\widetilde{\delta} \simeq \delta\otimes \delta_1$, $\widetilde{\sigma\otimes \delta} \simeq \sigma\otimes \delta\otimes \rho_{0, (\sigma, \delta)}$, as projective representations. Assume $\sigma$, $\delta$, $\sigma_1$, $\delta_1$, $\rho_{0, (\sigma, \delta)}$ are $\alpha$-projective, $\beta$-projective, $\alpha^{-1}$-projective, $\beta^{-1}$-projective, $\alpha^{-1}\times \beta^{-1}$-projective representations respectively. Since $m_{H_1\times H_2}(\rho_0, \sigma\otimes \delta)=1$, $m_{H_1\times H_2}(\widetilde{\sigma\otimes \delta}, \sigma\otimes \delta)=1$ and $\dim \rho_{0, (\sigma, \delta)}=1$.

 Hence $\Hom_{\Gamma}(\rho_0, \pi_1\otimes \pi_2)\simeq  \Hom_{\Gamma}(\Ind^{\Gamma}_{I_{\Gamma}(\sigma\otimes \delta)}\widetilde{\sigma\otimes \delta}, \Ind_{I_{G_1}(\sigma)}^{G_1} \widetilde{\sigma} \otimes  \Ind_{I_{G_2}(\delta)}^{G_2} \widetilde{\delta})\simeq \Hom_{I_{\Gamma}(\sigma\otimes \delta)}(\widetilde{\sigma\otimes \delta}, \Ind_{I_{G_1}(\sigma)}^{G_1} \widetilde{\sigma} \otimes  \Ind_{I_{G_2}(\delta)}^{G_2} \widetilde{\delta}) \simeq \Hom_{I_{\Gamma}(\sigma\otimes \delta)}(\widetilde{\sigma\otimes \delta}, \widetilde{\sigma} \otimes \widetilde{\delta}) \simeq [\Hom_{H_1\times H_2}(\widetilde{\sigma\otimes \delta}, \sigma\otimes \sigma_1\otimes \delta\otimes \delta_1)]^{\frac{I_{\Gamma}(\sigma\otimes \delta)}{H_1\times H_2}} \simeq  [\Hom_{H_1\times H_2}(\widetilde{\sigma\otimes \delta}, \sigma\otimes  \delta) \otimes (\sigma_1\otimes  \delta_1)]^{\frac{I_{\Gamma}(\sigma\otimes \delta)}{H_1\times H_2}}$. Note that $\frac{I_{\Gamma}(\sigma\otimes \delta)}{H_1\times H_2}$ acts projectively on  $\Hom_{H_1\times H_2}(\widetilde{\sigma\otimes \delta}, \sigma\otimes  \delta)$, which is of dimension one.   Hence by Lmm.\ref{isom}, (1) $\sigma_1$ is projectively isomorphic to $ \overline{\delta_1}\circ \gamma$, where $\gamma: \frac{I_{G_1}(\sigma)}{H_1}\simeq \frac{I_{G_2}(\delta)}{H_2}$; (2) $\delta_1$ linearly depends on $\sigma_1$ and  the action of $\frac{I_{\Gamma}(\sigma\otimes \delta)}{H_1\times H_2}$ on $\Hom_{H_1\times H_2}(\widetilde{\sigma\otimes \delta}, \sigma\otimes  \delta)$. Moreover, $m_{\Gamma}(\rho_0, \pi_1\otimes \pi_2)=1$.

Assume $\pi_1\otimes \pi_2|_{\Gamma} \simeq \rho_0\oplus (\oplus_{i=1}^n\rho_i)$, for some $\rho_i\in \widehat{\Gamma}$. If for some other $i \geq 1$,  $\rho_i \in \mathcal{R}_{\Gamma}(\rho)\cap  \mathcal{R}_{\Gamma}( \pi_1\otimes \pi_2)$.  Assume $\sigma^{g_1} \otimes \delta^{g_2} \in \mathcal{R}_{H_1\times H_2}(\rho_i)\cap  \mathcal{R}_{H_1\times H_2}( \pi_1\otimes \pi_2)$. Assume $(g_1^{-1}, \gamma(g_1^{-1})) \in \Gamma$.  Then $\sigma \otimes \delta^{g_2\gamma(g_1^{-1})} \simeq (\sigma^{g_1} \otimes \delta^{g_2})^{(g_1^{-1}, \gamma(g_1^{-1}))} \in \mathcal{R}_{H_1\times H_2}(\rho_i) \subseteq \mathcal{R}_{H_1\times H_2}(\rho)$. Since $\rho|_{H_1\times H_2}$ is a theta representation, $\delta^{g_2\gamma(g_1^{-1})} \simeq \delta$, and $\sigma\otimes \delta \in \mathcal{R}_{H_1\times H_2}(\rho_i)$. Hence $m_{H_1\times H_2}(\rho, \sigma\otimes \delta) \geq m_{H_1\times H_2}(\rho_0, \sigma\otimes \delta)+m_{H_1\times H_2}(\rho_i, \sigma\otimes \delta) \geq 2$, a contradiction. Therefore $\mathcal{R}_{\Gamma}(\rho)\cap  \mathcal{R}_{\Gamma}( \pi_1\otimes \pi_2)=\{ \rho_0\}$.  Since $m_{\Gamma}(\rho, \rho_0)=1=m_{\Gamma}(\rho_0, \pi_1\otimes \pi_2)$,  $m_{\Gamma}(\rho, \pi_1\otimes \pi_2)=1$, which implies that $m_{G_1\times G_2}(\pi, \pi_1\otimes \pi_2)=1$.

If $\pi_1\otimes \pi_2' \in  \mathcal{R}_{G_1\times G_2}(\pi)$, $0\neq \Hom_{G_1\times G_2}(\pi, \pi_1\otimes \pi'_2) \simeq \Hom_{\Gamma}(\rho, \pi_1\otimes \pi'_2)$. Hence $ \Hom_{H_1\times H_2}(\rho, \pi_1\otimes \pi'_2)\neq 0 $,  then $\Hom_{H_1\times H_2}(\rho, \sigma'\otimes \pi'_2)\neq 0$, for some $\sigma' \simeq \sigma^{g_1}\leq \pi_1|_{H_1}$;  it  implies $\Hom_{H_1\times H_2}(\rho, \sigma\otimes \pi'_2)\neq 0 $. Since $\rho$ is a theta representation,  $\delta \in  \mathcal{R}_{H_2}(\pi'_2)$.  Similar to the above proof, $\pi_2'\simeq \Ind_{I_{G_2}(\delta)}^{G_2} \widetilde{\delta}'$, with $\widetilde{\delta}' \simeq  \delta\otimes \delta_1'$ as projective representations. Moreover, by the above proof and Lmm.\ref{isom}, through $\sigma_1$, we can derive that  $\delta_1'$ is linearly equivalent to $ \delta_1$. By Thm.\ref{cliffordadmissible2}(4), $\widetilde{\delta} \simeq \widetilde{\delta}'$ and then $\pi_2\simeq \pi_2'$.  By duality, $\pi_1$ is also determined by $\pi_2$ uniquely.
\subsubsection{$(\Leftarrow)$}
Assume $\sigma\otimes \delta\in \mathcal{R}_{H_1\times H_2}(\rho)$. Note that $G_1\times G_2=\Gamma (G_1\times H_2)=\Gamma (H_1\times G_2)$. Then  $\Res^{G_1\times G_2}_{G_1\times H_2}\pi \simeq \Ind_{H_1\times H_2}^{G_1\times H_2} \rho$,  $\Res^{G_1\times G_2}_{H_1\times G_2}\pi \simeq \Ind_{H_1\times H_2}^{H_1\times G_2} \rho$. Assume $\sigma \in \mathcal{R}_{H_1}(\pi_1)$, for some $(\pi_1, V_1) \in \widehat{G_1}$. By Thm.\ref{HSSP},
\begin{equation}\label{eq1}
\begin{aligned}
0\neq \HS-\Hom_{H_1\times H_2}(\rho, \pi_1\otimes \delta)\simeq \HS-\Hom_{G_1\times H_2}( \Ind_{H_1\times H_2}^{G_1\times H_2} \rho, \pi_1\otimes\delta)\\
 \simeq \HS-\Hom_{G_1\times H_2}( \Res^{G_1\times G_2}_{G_1\times H_2}\pi, \pi_1\otimes\delta) \simeq \HS-\Hom_{G_1\times G_2}( \pi, \pi_1\otimes \Ind_{H_2}^{G_2}\delta).
 \end{aligned}
 \end{equation}
 So there exists a unique $\pi_2\in \mathcal{R}_{G_2}(\Ind_{H_2}^{G_2}\delta)$, such that $\pi_1\otimes \pi_2\in \mathcal{R}_{G_1\times G_2}(\pi)$. Certainly, $\delta\in \mathcal{R}_{H_2}(\pi_2)$.  Assume $n_{\delta}=m_{H_2}(\pi_2, \delta)$. Then $m_{H_1\times H_2}(\rho, \sigma\otimes \delta) \leq m_{H_1\times H_2}(\rho, \pi_1\otimes \delta) =m_{G_1\times G_2}( \pi, \pi_1\otimes \Ind_{H_2}^{G_2}\delta)=m_{G_1\times G_2}( \pi, n_{\delta}\pi_1\otimes \pi_2)=n_{\delta} <+\infty$. Hence $\Res_{H_1\times H_2}^{\Gamma} \rho$ is an admissible representation.  Similarly,
 \begin{equation}\label{eq2}
 0\neq \HS-\Hom_{H_1\times H_2}(\rho,  \sigma \otimes \pi_2)  \simeq \HS-\Hom_{G_1\times G_2}( \pi, \Ind_{H_1}^{G_1}\sigma\otimes \pi_2).
 \end{equation}
  Let  us write $n_{\sigma}=m_{H_1}(\sigma, \pi_1)$, and $m=m_{H_1\times H_2}(\rho, \sigma\otimes \delta)$. Then by (\ref{eq1}), $n_{\delta}\geq n_{\sigma}m$, and by (\ref{eq2}), $n_{\sigma}\geq n_{\delta}m$. Since $m, n_{\sigma}, n_{\delta}$ all are finite positive integers, $m=1$ and $n_{\sigma}=n_{\delta}$.  Go back to (\ref{eq1}), $\mathcal{R}_{H_1\times H_2}(\rho) \cap \mathcal{R}_{H_1\times H_2}( \pi_1\otimes \delta)=\{ \sigma\otimes \delta\}$. If $\sigma'\otimes \delta \in \mathcal{R}_{H_1\times H_2}(\rho)$, then $0\neq \HS-\Hom_{H_1\times H_2}(\rho,  \sigma' \otimes \pi_2)  \simeq \HS-\Hom_{G_1\times G_2}( \pi, \Ind_{H_1}^{G_1}\sigma'\otimes \pi_2)$. Hence $\pi_1\in \mathcal{R}_{G_1}(\Ind_{H_1}^{G_1}\sigma')$, and  $\sigma' \in \mathcal{R}_{H_1}(\pi_1)$. Go back to (\ref{eq1}), $\sigma'\simeq \sigma$.  In other words, $\sigma$ is uniquely determined by $\delta$. By duality, we can obtain the whole result.$\Box$\\

  Let $\Omega_i(-,-)$ be a $2$-cocycle representing some class in $\Ha_m^2(G_i, \T)$, for each $i=1,2$. Let $\omega_i=\Omega_i|_{H_i\times H_i}$.  Let $\Omega=\Omega_1\times \Omega_2 $, $\omega=\omega_1\times \omega_2$, and $\omega_{\Gamma}=\Omega|_{\Gamma}$.  Assume now $(\rho, W)$ is an $\omega_{\Gamma}$-projective representation of $\Gamma$.
\begin{corollary}\label{twoext}
$\Res_{H_1\times H_2,\omega}^{\Gamma, \omega_{\Gamma}} \rho$ is an $\omega$-projective theta representation of $H_1\times H_2$ iff $\Ind_{\Gamma,\omega_{\Gamma}}^{G_1\times G_2,\Omega} \rho$  is an $\Omega$-projective theta representation of $G_1\times G_2$.
\end{corollary}
\begin{proof}
We can lift both sides to the corresponding ordinary representations. In that case, the result has been proved above.  More precisely, let $H_i^{\omega_i}$, $G_i^{\Omega_i}$ be the lifting groups.  It can be seen that $H_i^{\omega_i}$ is a normal subgroup of $G_i^{\Omega_i}$. Now let $\gamma^{\Omega}$ be the composite map: $\frac{G^{\Omega_1}_1}{H^{\omega_1}_1} \simeq \frac{G_1}{H_1} \stackrel{\gamma}{\simeq } \frac{G_2}{H_2} \simeq \frac{G^{\Omega_2}_2}{H^{\omega_2}_2}$. Let us write  the graph $\Gamma(\gamma^{\Omega})\simeq \frac{\widetilde{\Gamma^{\omega_{\Gamma}}}}{H^{\omega_1}_1\times H^{\omega_2}_2}$.  Then there exists a short exact sequence
$$1\longrightarrow \T \stackrel{\iota}{\longrightarrow} \widetilde{\Gamma^{\omega_{\Gamma}}} \stackrel{\kappa}{\longrightarrow} \Gamma^{\omega_{\Gamma}} \longrightarrow 1.$$
Let $\rho^{\omega_{\Gamma}}$ be the lifting representation of $\Gamma^{\omega_{\Gamma}}$ from $\rho$. Through $\kappa$, let $\widetilde{\rho^{\omega_{\Gamma}}}=\rho^{\omega_{\Gamma}} \circ \kappa$.   Similarly, $1\longrightarrow \T \stackrel{\iota}{\longrightarrow} H_1^{\omega_1} \times H_2^{\omega_2} \stackrel{\kappa}{\longrightarrow} (H_1\times H_2)^{\omega} \longrightarrow 1$,  $1\longrightarrow \T \stackrel{\iota}{\longrightarrow} G_1^{\Omega_1} \times G_2^{\Omega_2} \stackrel{\kappa}{\longrightarrow} (G_1\times G_2)^{\Omega} \longrightarrow 1$ both are  exact sequences of groups. Let $\kappa': \frac{  G_1^{\Omega_1} \times G_2^{\Omega_2} }{\T} \simeq(G_1\times G_2)^{\Omega}$.   Moreover, $\Ind^{G_1^{\Omega_1}\times G_2^{\Omega_2}}_{\widetilde{\Gamma^{\omega_{\Gamma}}}} \widetilde{\rho^{\omega_{\Gamma}}} \simeq [\Ind^{(G_1\times G_2)^{\Omega}}_{\Gamma^{\omega_{\Gamma}}} \rho^{\omega_{\Gamma}}]\circ \kappa'$.
 By the above theorem, $\Res_{H_1^{\omega_1}\times H_2^{\omega_2}}^{\widetilde{\Gamma^{\omega_{\Gamma}}}} \widetilde{\rho^{\omega_{\Gamma}}}$ is a $\theta$-representation iff $\Ind^{G_1^{\Omega_1}\times G_2^{\Omega_2}}_{\widetilde{\Gamma^{\omega_{\Gamma}}}} \widetilde{\rho^{\omega_{\Gamma}}}$  is a $\theta$-representation. This will imply the result.
\end{proof}
\subsubsection{Related to Rieffel Equivalence}
 Keep the above notations.   Recall Rieffel Equivalence from section \ref{RE}. For simplicity, let us write $H=H_1\times H_2$, $G=G_1\times G_2$.

 \begin{lemma}
 For $\rho'\in \widehat{\Gamma}$, $\mathcal{R}_{G}(\Ind_{\Gamma}^{G}\rho')$ (resp. $\mathcal{R}_{G_i}(\Ind_{\Gamma}^{G}\rho')$, resp. $\mathcal{R}_{H}(\rho')$, resp.$\mathcal{R}_{H_i}(\rho')$) only contains one equivalence class in  $\frac{\widehat{G}}{\sim_{(G, H)}}$ (resp. $\frac{\widehat{G_i}}{\sim_{(G_i, H_i)}}$, resp. $\frac{\widehat{H}}{\sim_{(G, H)}}$, resp. $\frac{\widehat{H_i}}{\sim_{(G_i, H_i)}}$), $i=1, 2$.
 \end{lemma}
 \begin{proof}
 1) For $\sigma\otimes \delta\in  \mathcal{R}_{H}(\rho')$, $\rho'  \hookrightarrow \Ind_H^{\Gamma} (\sigma \otimes \delta)$,  so $\Ind_{\Gamma}^{G}\rho' \hookrightarrow \Ind_H^{G} (\sigma \otimes \delta)$. Hence $\mathcal{R}_{G}(\Ind_{\Gamma}^{G}\rho') \subseteq \mathcal{R}_{G}(\Ind_H^{G} (\sigma \otimes \delta))$,   $\mathcal{R}_{G_1}(\Ind_{\Gamma}^{G}\rho') \subseteq \mathcal{R}_{G_1}(\Ind_{H_1}^{G_1} \sigma)$,  $\mathcal{R}_{G_2}(\Ind_{\Gamma}^{G}\rho') \subseteq \mathcal{R}_{G_2}(\Ind_{H_2}^{G_2} \delta)$. So the first three statements hold.\\
 2) Assume $\pi_1\otimes \pi_2 \in  \mathcal{R}_{G}(\Ind_{\Gamma}^{G}\rho')$. Then $\rho'  \hookrightarrow \Res_{\Gamma}^G (\pi_1\otimes \pi_2)$.  For any  $\sigma\otimes \delta\in  \mathcal{R}_{H}(\rho')$, $ \sigma\otimes \delta \hookrightarrow \Res^{\Gamma}_H \rho' \hookrightarrow \Res^{G}_H (\pi_1\otimes \pi_2)$. Hence the last three statements hold.
 \end{proof}
  \begin{proposition}
 For $(\rho, W)\in \Rep(\Gamma)$, $\Res_{H}^{\Gamma} \rho$ is a bigraphic   representation of $H$ with respect to $\frac{\widehat{H}}{\sim_{(G,H)}}$ iff $\Ind_{\Gamma}^{G} \rho$  is a   bigraphic representation of $G$ with respect to $\frac{\widehat{G}}{\sim_{(G, H)}}$.
\end{proposition}
\begin{proof}
Assume $\rho\simeq \widehat{\oplus}_{i\in I} \rho_i$, for $\rho_i \in \widehat{\Gamma}$. By the above lemma, for each  $\rho_i$,  assume   $\mathcal{R}_{H_1}( \rho_i)\subseteq [\sigma_i]\in \frac{\widehat{H_1}}{\sim_{(G_1, H_1)}}$, $\mathcal{R}_{H_2}( \rho_i)\subseteq [\delta_i]\in \frac{\widehat{H_2}}{\sim_{(G_2, H_2)}}$, and then  $\mathcal{R}_{H}( \rho_i)\subseteq [\sigma_i\otimes \delta_i]\in \frac{\widehat{H}}{\sim_{(G, H)}}$. \\
Similarly, assume  $\mathcal{R}_{G_1}( \Ind_{\Gamma}^{G}\rho_i)\subseteq [\pi^{(1)}_i]\in \frac{\widehat{G_1}}{\sim_{(G_1, H_1)}}$, $\mathcal{R}_{G_2}( \Ind_{\Gamma}^{G} \rho_i)\subseteq [\pi^{(2)}_i]\in \frac{\widehat{G_2}}{\sim_{(G_2, H_2)}}$, and then  $\mathcal{R}_{G}(\Ind_{\Gamma}^{G} \rho_i)\subseteq [\pi^{(1)}_i\otimes\pi^{(2)}_i]\in \frac{\widehat{G}}{\sim_{(G, H)}}$. \\
Moreover, $[\sigma_i]$ (resp. $[\delta_i]$) corresponds to $[\pi_i^{(1)}]$ (resp. $[\pi_i^{(2)}]$) by using the functors $\Ind_{H_1}^{G_1}$, $\Res_{H_1}^{G_1}$ ( resp. $\Ind_{H_2}^{G_2}$, $\Res_{H_2}^{G_2}$). This implies  that $[\sigma_i] \neq [\sigma_j]$ iff $ [\pi_i^{(1)}] \neq [\pi_j^{(1)}]$, and $[\delta_i] \neq [\delta_j]$ iff $ [\pi_i^{(2)}] \neq [\pi_j^{(2)}]$. Therefore the result holds.
\end{proof}
\subsubsection{Abelian case}\label{abeliancase}
 Keep the notations, and also assume  $(\rho, W) \in \Rep(\Gamma)$. Assume  $\frac{G_1}{H_1}\simeq \frac{G_2}{H_2}$, and both are  \textbf{abelian} groups. Firstly let us present some lemmas. Recall  $p: G_1\times G_2 \longrightarrow \frac{G_1}{H_1} \times \frac{G_2}{H_2}$. For $g_i\in G_i$, $p((g_1, g_2)\Gamma (g_1,g_2)^{-1})=p(\Gamma)$, so $ (g_1, g_2)\Gamma (g_1,g_2)^{-1} \subseteq \Gamma$. Dually,  $\Gamma$ is a normal subgroup of $G_1\times G_2$.  Moreover, $\frac{G_1\times G_2}{\Gamma} \simeq \frac{\Gamma}{H_1\times H_2} \simeq \frac{G_i}{H_i}$. Follow the notations of Lmm.\ref{GHEQ}. For simplicity, let us write $H=H_1\times H_2$, $G=G_1\times G_2$.

Without the consideration of the multiplicity, we obtain the next two results:
\begin{proposition}\label{Second}
For  $\rho \in \widehat{\Gamma}$, assume:
 \begin{itemize}
\item[(1)]  $\Res_{H}^{\Gamma} \rho$  is a bigraphic representation of $H$,
\item[(2)] for each $\sigma\in \mathcal{R}_{H_1}(\rho)$, $ \Ind_{H_1}^{G_1} \sigma$ is multiplicity-free.
\end{itemize}
Then  $\pi=\Ind_{\Gamma}^{G} \rho$ is a $\theta_1$-graphic  representation of $G$.
\end{proposition}

 \begin{proof}Assume $\pi_1\otimes \pi_2\in \mathcal{R}_{G}(\Ind_{\Gamma}^{G} \rho)$, and $\sigma\otimes \delta\in \mathcal{R}_{H}(\pi_1\otimes \pi_2) \cap \mathcal{R}_{H}(\rho)$.  Then $\gamma: \frac{I_{G_1}(\sigma)}{H_1} \simeq \frac{I_{G_2}(\delta)}{H_2}$ with the graph $\Gamma(\gamma)_{(\sigma, \delta)}= [(I_{G_1}(\sigma) \times I_{G_2}(\delta)) \cap\Gamma]/H$. Moreover, $ [(I_{G_1}(\sigma) \times I_{G_2}(\delta)) \cap\Gamma]=I_{\Gamma}(\sigma\otimes \delta)$.
 Assume $\pi_1\simeq \Ind^{G_1}_{I_{G_1}(\sigma)}\widetilde{\sigma}$, $\pi_2\simeq \Ind^{G_2}_{I_{G_2}(\delta)}\widetilde{\delta}$, $\rho\simeq \Ind^{\Gamma}_{I_{\Gamma}(\sigma\otimes \delta)}\widetilde{\sigma\otimes \delta}$.  Hence:
\begin{equation}\label{eqthm1}
\begin{aligned}
0\neq \HS-\Hom_{G}(\Ind_{\Gamma}^G \rho, \pi_1\otimes \pi_2) \simeq \HS-\Hom_{G}(\Ind_{I_{\Gamma}(\sigma\otimes \delta)}^G \widetilde{\sigma\otimes \delta}, \pi_1\otimes \pi_2)
\\ \simeq \HS-\Hom_{I_{\Gamma}(\sigma\otimes \delta)}( \widetilde{\sigma\otimes \delta}, \pi_1\otimes \pi_2)
 \simeq \HS-\Hom_{I_{\Gamma}(\sigma\otimes \delta)}( \widetilde{\sigma\otimes \delta}, \Ind_{I_{G_1}(\sigma)}^{G_1}\widetilde{\sigma}\otimes \Ind_{I_{G_2}(\sigma)}^{G_2}\widetilde{\delta}) \\
\simeq \HS-\Hom_{I_{\Gamma}(\sigma\otimes \delta)}( \widetilde{\sigma\otimes \delta},\widetilde{\sigma}\otimes \widetilde{\delta}).
 \end{aligned}
 \end{equation}

  By the condition (2),  let us  write  $\widetilde{\sigma}=  \sigma$,  and write $\widetilde{\delta} = \delta\otimes \delta_1$,  $\widetilde{\sigma\otimes \delta}=(\sigma \otimes \delta) \otimes \rho_{1}$  as projective representations. By (\ref{eqthm1}), $\widetilde{\sigma\otimes \delta}$ is a subrepresentation of $(\widetilde{\sigma}\otimes \widetilde{\delta})|_{I_{\Gamma}(\sigma\otimes \delta)}$. Hence $\rho_1$ is  a linearly isomorphic to a  subrepresentation of $(1\otimes \delta_1)|_{I_{\Gamma}(\sigma\otimes \delta)}$. Note that $(1\otimes \delta_1)|_{I_{\Gamma}(\sigma\otimes \delta)}$ is  irreducible. Hence $\rho_1$ is  linearly isomorphic to $(1\otimes \delta_1)|_{I_{\Gamma}(\sigma\otimes \delta)}$.

   If $\pi_1\otimes \pi_2' \in \mathcal{R}_{G}(\Ind_{\Gamma}^{G} \rho)$, then $\sigma'\otimes \delta'\in  \mathcal{R}_{H}(\pi_1\otimes \pi_2') \cap \mathcal{R}_{H}(\rho)$. Assume $\sigma\simeq \sigma'^g$, for some $g\in G_1$, and assume $(g, \gamma(g))\in \Gamma$. Then $\sigma\otimes \delta'^{\gamma(g)} \simeq \sigma'^g\otimes \delta'^{\gamma(g)}\in\mathcal{R}_{H}(\pi_1\otimes \pi_2') \cap \mathcal{R}_{H}(\rho)$. By graphic property, $\delta'^{\gamma(g)}\simeq \delta$. Hence $\delta\in \mathcal{R}_{H_2}(\pi'_2)$.

   Assume   $\pi_2\simeq \Ind^{G_2}_{I_{G_2}(\delta)}\widetilde{\delta}'$, and $\widetilde{\delta}'= \delta \otimes \delta_1'$ as projective representations. Similar to the above discussion,
   $\rho_1$ is linearly isomorphic to  $ (1\otimes \delta'_1)|_{I_{\Gamma}(\sigma\otimes \delta)}$. Therefore $ \delta_1 $ is linearly isomorphic to $\delta'_1$. By Clifford-Mackey theory, $\widetilde{\delta} \simeq \widetilde{\delta}'$  as $I_{G_2}(\delta)$-modules, so  $\pi_2\simeq \pi_2'$;   this is what we need to prove.
   \end{proof}

\begin{proposition}\label{thirdre}
For  $\rho \in \widehat{\Gamma}$, assume:
 \begin{itemize}
\item[(1)] $\pi=\Ind_{\Gamma}^{G} \rho$  is a bigraphic representation of $G$,
\item[(2)] for each $\pi_1\in \mathcal{R}_{H_1}(\pi)$, $ \pi_1|_{H_1}$ is multiplicity-free.
\end{itemize}
Then  $\Res_{H}^{\Gamma} \rho$ is a $\theta_1$-graphic  representation of $H$.
\end{proposition}
\subsubsection{Proof of the proposition  \ref{thirdre}}
I) Assume $\sigma\otimes \delta\in \mathcal{R}_{H_1\times H_2}(\rho)$, $\sigma \in \mathcal{R}_{H_1}(\pi_1)$, for some $(\pi_1, V_1) \in \widehat{G_1}$. By Frobenius reciprocity, we have:
\begin{equation}\label{coroeq1}
\begin{aligned}
0\neq \HS-\Hom_{H_1\times H_2}(\rho, \pi_1\otimes \delta)\simeq \HS-\Hom_{G_1\times H_2}( \Ind_{H_1\times H_2}^{G_1\times H_2} \rho, \pi_1\otimes\delta)\\
 \simeq\HS-\Hom_{G_1\times H_2}( \Res^{G_1\times G_2}_{G_1\times H_2}\pi, \pi_1\otimes\delta) \simeq \HS-\Hom_{G_1\times G_2}( \pi, \pi_1\otimes \Ind_{H_2}^{G_2}\delta).
 \end{aligned}
 \end{equation}
 Hence there exists a unique $\pi_2\in \mathcal{R}_{G_2}(\Ind_{H_2}^{G_2}\delta)$, such that $\pi_1\otimes \pi_2\in \mathcal{R}_{G_1\times G_2}(\pi)$. Certainly, $\delta\in \mathcal{R}_{H_2}(\pi_2)$. Similarly,
 \begin{equation}\label{coroeq2}
 0\neq \HS-\Hom_{H_1\times H_2}(\rho,  \sigma \otimes \pi_2)  \simeq \HS-\Hom_{G_1\times G_2}( \pi, \Ind_{H_1}^{G_1}\sigma\otimes \pi_2).
 \end{equation}
 \begin{equation}\label{eq5}
\HS-\Hom_{H_1\times H_2}(\rho,  \pi_1\otimes \pi_2)  \simeq \HS-\Hom_{G_1\times G_2}( \pi, \Ind_{H_1}^{G_1}\pi_1\otimes \pi_2) \simeq \HS-\Hom_{G_1\times G_2}( \pi, \pi_1\otimes\Ind_{H_2}^{G_2} \pi_2)
 \end{equation}
  Assume $s_1=[G_1: I_{G_1}(\sigma)]$, $s_2=[G_2: I_{G_2}(\delta)]$, $\mathcal{R}_{H_1}(\pi_1)=\{ \sigma=\sigma_1, \cdots, \sigma_{s_1}\}$, $\mathcal{R}_{H_2}(\pi_2)=\{ \delta=\delta_1, \cdots, \delta_{s_2}\}$, $\sigma_i\simeq \sigma^{g^{(1)}_i}$, $\delta_j\simeq \delta^{g^{(2)}_j}$, for some $g^{(1)}_i \in G_1$, $g^{(2)}_j\in G_2$. For $\rho$, write  $\rho_{\delta}=\Theta_{\delta} \otimes \delta$,  $\rho_{\sigma}\simeq \sigma \otimes \Theta_{\sigma}$, as $H_1\times H_2$-modules. For simplicity,   assume $\mathcal{R}_{H_1}(\Theta_{\delta})=\{ \sigma=\sigma_1, \sigma_2, \cdots, \sigma_{t_1}\}$, $\mathcal{R}_{H_2}(\Theta_{\sigma})=\{ \delta=\delta_1, \delta_2, \cdots, \delta_{t_2}\}$. Let  us write $n_{\sigma}=m_{H_1}(\sigma, \pi_1)$, $n_{\delta}=m_{H_2}(\delta, \pi_2)$,  $m=m_{H_1\times H_2}(\rho, \sigma\otimes \delta)$, $t=m_{G_1\times G_2}(\pi, \pi_1\otimes \pi_2)$. Then by (\ref{coroeq1}), $n_{\delta}t= t_1n_{\sigma}m$;  by (\ref{coroeq2}), $n_{\sigma}t=t_2 n_{\delta}m$; by  (\ref{eq5}), $n_{\sigma}n_{\delta} m t_2 s_1=n_{\sigma}n_{\delta} m t_1 s_2 =tn_{\sigma}^2s_1=tn_{\delta}^2s_2$. These equalities imply that $\frac{n_{\sigma}}{n_{\delta}}= \frac{t}{t_1m}=\frac{t_2m}{t}=\sqrt{\frac{t_2}{t_1}}=\sqrt{\frac{s_2}{s_1}} \cdots (1)$.\\
 II) Assume $\gamma: \frac{I_{G_1}(\sigma)}{H_1} \simeq \frac{I_{G_2}'(\delta)}{H_2}$,  $\gamma: \frac{I'_{G_1}(\sigma)}{H_1} \simeq \frac{I_{G_2}(\delta)}{H_2}$. Then $\gamma:  \frac{I_{G_1}(\sigma)\cap I'_{G_1}(\sigma)}{H_1} \simeq \frac{I_{G_2}(\delta)\cap I_{G_2}'(\delta)}{H_2}$.
 Note that $I_{\Gamma}(\sigma\otimes \delta)=\{ (g_1,g_2) \in \Gamma \mid \sigma^{g_1} \simeq \sigma, \delta^{g_2}\simeq \delta\}=\Gamma \cap [I_{G_1}(\sigma)\cap I'_{G_1}(\sigma) \times I_{G_2}(\delta)\cap I_{G_2}'(\delta)] $. Assume $\pi_1\simeq \Ind^{G_1}_{I_{G_1}(\sigma)}\widetilde{\sigma}$, $\pi_2\simeq \Ind^{G_2}_{I_{G_2}(\delta)}\widetilde{\delta}$, $\rho\simeq \Ind^{\Gamma}_{I_{\Gamma}(\sigma\otimes \delta)}\widetilde{\sigma\otimes \delta}$. Assume $\widetilde{\sigma}\simeq \sigma\otimes \sigma_1$, $\widetilde{\delta} \simeq \delta\otimes \delta_1$, $\widetilde{\sigma\otimes \delta} \simeq \sigma\otimes \delta\otimes \rho_1$, as projective representations. Assume $\sigma$, $\delta$, $\sigma_1$, $\delta_1$, $\rho_1$ are $\alpha$-projective, $\beta$-projective, $\alpha^{-1}$-projective, $\beta^{-1}$-projective, $\alpha^{-1}\times \beta^{-1}$-projective representations respectively. Notice that $\sigma_i \otimes \delta \in \mathcal{R}_{H}(\rho)$ iff $\exists (h_i, \gamma(h_i))\in \Gamma$ such that $\sigma_i \otimes \delta  \simeq \sigma^{h_i}\otimes \delta^{\gamma(h_i)}$ iff $\gamma(h_i) \in I_{G_2}(\delta)$ or $h_i \in  I'_{G_1}(\sigma)$. Also for two such $h_i, h_j$, $\sigma^{h_i} \simeq \sigma^{h_j}$ iff $h_i h_j^{-1} \in I_{G_1}(\sigma)\cap I'_{G_1}(\sigma)$. Hence $t_1=[I'_{G_1}(\sigma): I_{G_1}(\sigma)\cap I'_{G_1}(\sigma)]=[I_{G_2}(\delta): I_{G_2}(\delta)\cap I'_{G_2}(\delta)]$.  Dually, $t_2=[I'_{G_2}(\delta): I_{G_2}(\delta)\cap I'_{G_2}(\delta)]=[I_{G_1}(\sigma): I_{G_1}(\sigma)\cap I'_{G_1}(\sigma)]$.\\
III) Consider now \begin{equation}\label{coroeq7}
\begin{aligned}
&\HS-\Hom_G(\pi, \pi_1\otimes \pi_2) \simeq \HS-\Hom_{G}(\Ind_{I_{\Gamma}(\sigma\otimes \delta)}^G \widetilde{\sigma\otimes \delta}, \pi_1\otimes \pi_2)\\
& \simeq \HS-\Hom_{I_{\Gamma}(\sigma\otimes \delta)}(\widetilde{\sigma\otimes \delta}, \pi_1\otimes \pi_2) \simeq \HS-\Hom_{I_{\Gamma}(\sigma\otimes \delta)}(\widetilde{\sigma\otimes \delta}, \widetilde{\sigma}\otimes \widetilde{\delta})\\
 &\simeq \HS-\Hom_{I_{G_1}(\sigma) \cap I'_{G_1}(\sigma) \times I_{G_2}(\delta) \cap I'_{G_2}(\delta)}(\Ind^{I_{G_1}(\sigma) \cap I'_{G_1}(\sigma) \times I_{G_2}(\delta) \cap I'_{G_2}(\delta)}_{I_{\Gamma}(\sigma\otimes \delta)}\widetilde{\sigma\otimes \delta}, \widetilde{\sigma}\otimes \widetilde{\delta}).
 \end{aligned}
 \end{equation}
Since $\pi$ is a bigraphic representation, $\widetilde{\sigma}\simeq \widetilde{\sigma} \otimes \chi_1$, for any $\chi_1 \in \widehat{(\frac{I_{G_1}(\sigma)}{I_{G_1}(\sigma) \cap I'_{G_1}(\sigma)})}$, and $\widetilde{\delta}\simeq \widetilde{\delta} \otimes \chi_2$, for any $\chi_2 \in \widehat{(\frac{I_{G_2}(\delta)}{I_{G_2}(\delta) \cap I'_{G_2}(\delta)})}$. Applying the result of Lmm.\ref{eqeq}, we obtain that $[I_{G_1}(\sigma):I_{G_1}(\sigma) \cap I'_{G_1}(\sigma)]=f_1e_1^2$, and similarly $[I_{G_2}(\delta):I_{G_2}(\delta) \cap I'_{G_2}(\delta)]=f_2e_2^2$. Hence $t_1=f_2e_2^2$, $t_2=f_1e^2_1$.
By Lmm.\ref{CHI2}, $t_2\leq n_{\sigma}^2$, $t_1\leq n_{\delta}^2$.  Assume irreducible $\widetilde{\sigma}'_1 \leq \widetilde{\sigma}|_{I_{G_1}(\sigma) \cap I'_{G_1}(\sigma)}  $, and irreducible $\widetilde{\delta}'_1 \leq \widetilde{\delta}|_{I_{G_2}(\delta) \cap I'_{G_2}(\delta)}$ such that $\HS-\Hom_{I_{\Gamma}(\sigma\otimes \delta)}(
\widetilde{\sigma\otimes \delta}, \widetilde{\sigma}'_1\otimes\widetilde{\delta}'_1) \neq 0$. Assume $\widetilde{\sigma}'_1 \simeq \sigma \otimes \sigma_1'$, $\widetilde{\delta}'_1 \simeq \delta \otimes \delta'_1$, as projective representations.
Then $\dim(\sigma_1')=\frac{n_{\sigma}}{f_1e_1}$, $\dim(\delta'_1)=\frac{n_{\delta}}{f_2e_2}$. \\
IV)  In case $n_{\sigma}=1$,  $f_1=e_1=t_2=1$. Hence $\mathcal{R}_{H_2}(\Theta_{\sigma})$ only contains one element $\delta$, and $\delta$ is determined directly by $\sigma$. So $\rho|_{H_1\times H_2}$ is a $\theta_1$-graphic representation.

 \subsection{}Finally  we can also get a result for finite index case. Keep the notations of the beginning subection \ref{abeliancase}.  Assume $n=\# \frac{G_i}{H_i}<+\infty $ and $n$ has no square factor.
\begin{corollary}
$\Res_{H}^{\Gamma} \rho$ is a bigraphic  representation of $H$ iff $\pi=\Ind_{\Gamma}^{G} \rho$  is a bigraphic representation of $G$.
\end{corollary}
\begin{proof}
It is a consequence of Coro.\ref{mutifree}, and Props. \ref{Second}, \ref{thirdre}.
\end{proof}
\section{Examples}

  \begin{example}\label{simple1}
Let  $G_1=G_2=\Gamma=G$.
 \begin{itemize}
\item[(1)] $\Ind_{G}^{G\times G} \chi$  is a theta representation of $G\times G$, for  any character $\chi$ of $G$;
\item[(2)]  $\Ind_{G}^{G\times G} \sigma$  is not a  theta representation of $G\times G$, for  $(\sigma, W)\in \widehat{G}$, $\dim W\geq  2$.
\end{itemize}
\end{example}

\begin{example}
Let $F$ be a finite field of order $q=p^n$, $p:$ an odd prime number. Let $(V, \langle, \rangle)$ be a symplectic vector space over $F$ of dimension $2m$, and $H(V)$ the corresponding Heisenberg  group. Let  $\psi$ be  a non-trivial character of $F$. Let $\pi_{\psi}$ be the Schr\"odinger representation of $H(V)$ associated to $\psi$. Then $V/0\simeq H(V)/F$, with the graph $\Gamma=\{ (v_h, h)\mid h=(v_h, t)\in H(V), v_h\in V, t\in F\}\simeq H(V)$.  Then $\Res^{\Gamma}_{0\times F }\pi_{\psi}=q^m (1\otimes \psi) $, which is a bigraphic representation, and $\Ind_{\Gamma}^{V \times H(V) } \pi_{\psi}= \oplus_{\chi\in \Irr(V)} \chi  \otimes \pi_{\psi} $, which is  a $\theta_1$-graphic  representation.
\end{example}
Let $S_n$ denote the symmetric group on $n$ letters $\{1, \cdots, n\}$.
\begin{example}
   Let us consider   the wreath product groups $G_1=S_{m_1} \wr S_n$, $G_2=S_{m_2} \wr S_n$, and their normal subgroups $H_1=\underbrace{S_{m_1}\times \cdots \times S_{m_1}}_{n}, H_2=\underbrace{S_{m_2}\times \cdots \times S_{m_2}}_n$. Then  $G_1/H_1\simeq G_2/H_2 \simeq S_n$ with the graphic group $\Gamma=(S_{m_1}\times S_{m_2})\wr S_n$.   Let $(\sigma_i, W_i)\in \Irr(S_{m_i})$, and $1$ the trivial representation of $S_n$. Then  $\sigma_i \wr 1\in \Irr(S_{m_i} \wr S_n)$, $(\sigma_1\otimes\sigma_2)\wr 1\in \Irr(\Gamma)$, and $[(\sigma_1\otimes\sigma_2)\wr 1]|_{H_1\times H_2} \simeq \sigma_1^{\otimes n}\otimes \sigma_2^{\otimes n}\in \Irr(H_1\times H_2)$. Then  $\Ind_{\Gamma}^{G_1\times G_2}[( \sigma_1\otimes\sigma_2)\wr 1] \simeq \oplus_{\pi\in \Irr(S_n)} (\sigma_1\wr \pi) \otimes (\sigma_2\wr \check{\pi})$, which  is a theta representation.
\end{example}

Let $\pi_1, \pi_2$ be two different irreducible representations of $S_n$. Let us consider the representation  $\pi=[\pi_1\oplus \pi_2]\wr 1= [\pi_1\oplus \pi_2]^{\otimes m} \in \Rep(S_n \wr S_m)$. According to \cite[p.57, Section 2.59]{K2},
$$\pi\simeq  \sum_{k=0}^m \Ind_{(S_n \wr S_k) \times (S_n \wr S_{m-k})}^{S_n \wr S_{m}} [(\pi_1\wr 1)\otimes (\pi_2 \wr 1)].$$
As a representation of $\underbrace{S_n\times \cdots \times S_n}_m$, every component of $[\pi_1\oplus \pi_2]\wr 1$ has the forms $\delta_1\otimes \cdots \otimes  \delta_m$, with $\delta_i=\pi_1$ or $\pi_2$.
\begin{example}
For a decomposition  $m=m_1+m_2$, $ m_1, m_2\geq 1$.  Let us reorganizer above components by using  theta representations as follows:\\ Let $\lambda_1=\{1, \cdots, m_1\}$, and $\lambda_2=\{m_1+1,\cdots, m=m_1+m_2\}$, and $S_{\lambda_1}\simeq S_{m_1}$, $S_{\lambda_2}\simeq S_{m_2}$.  For $0\leq k\leq m$, let $k=k_1+k_2$ with $0\leq k_1\leq m_1$, $0\leq k_2\leq m_2$.  For the  pair $(k_1, k_2)$, we associate to two representations $D_{k_1}=$ the sum of the  forms  of $k_1$-number $\pi_1$ and $(m_1-k_1)$-number $\pi_2$, and $D_{k_2}=$ the sum of the forms of  $k_2$-number $\pi_1$ and $(m_2-k_2)$-number $\pi_2$. Then
$$D_{k_1}\simeq \Ind_{(S_n \wr S_{k_1}) \times (S_n \wr S_{m_1-k_1})}^{S_n\wr S_{m_1}} [(\pi_1\wr 1) \otimes (\pi_2\wr 1)],$$
$$D_{k_2}\simeq \Ind_{(S_n \wr S_{k_2}) \times (S_n \wr S_{m_2-k_2})}^{S_n\wr S_{m_2}} [(\pi_1\wr 1) \otimes (\pi_2\wr 1)].$$
Let $\pi_k=\oplus_{k_1+k_2=k, 0\leq k_1\leq m_1, 0\leq k_2\leq m_2} D_{k_1}\otimes D_{k_2}$, which is a theta representation of $(S_n\wr S_{m_1} )\times (S_n\wr S_{m_2})$. Then $\Res^{S_n\wr S_m}_{(S_n \wr S_{m_1}) \times (S_n \wr S_{m_2})} \pi\simeq \sum_{k=0}^m\pi_k$.
\end{example}

\subsection{Examples}
In the last section, we apply the results into unitary representations of compact lie groups by  giving   some examples from reductive dual pairs. Here we only  give some examples of  representations of finite dimension.

Let us  recall the classification of irreducible  reductive dual pairs from \cite[Chap.1]{MVW} related to compact groups directly. For reductive dual pairs in symplectic groups or Lie groups, one can see  \cite{AdBa}, \cite{Ad1}, and \cite{Ru}, etc. Let $(D, \tau)=(\R, id),(\C, id), (\C, -)$\footnote{ Here, we miss the quaternion case for the reason  that we limit us to the compact group case.} be a field(commutative) with an involution $\tau$, $\epsilon=\pm 1$. Let $(W, \langle, \rangle_W)$ be a right $\epsilon$-Hermitian vector space over $D$ of dimension $n$. Let $H, H'$  ¡¡denote an irreducible reductive  dual pair (or called Howe dual pair) in $U(W)$.  To fit our purpose, let us modify $W$ into a complex vector space: \\
1. If $D=\C$, let $V=W$; \\
2. If $D=\R$, let $V=W\otimes_{\R}\C$.\\
 Let $(\omega, U(W), V)$ denote the associated representation(may not unitary). Let us write $\C=\R+i\R$, and $\mathbb{H}=\R+i\R+j\R+k\R$.
\begin{lemma}\label{doucom}
$\End_{H}(V)'=\End_{H'}(V)$ and $\End_{H'}(V)'=\End_{H}(V)$.
\end{lemma}
\begin{proof}
According to \cite[p.13, lemme]{MVW}, $\End_{DH}(W)$ is the commutant of $\End_{DH'}(W)$ in $\End_D(W)$, and vice-versa. Hence when $D=\C$, the result holds. \\
If $D=\R$, $\End_{\C}(V)=\End_{\R}(W)\otimes_{\R}\C\simeq \End_{\R}(W)\oplus i\End_{\R}(W)$. If $\phi=\phi_1+i\phi_2\in \End_{H}(V)$, then for  $h\in H$, $\phi h =h \phi$,  implies $\phi_i  h =h\phi_i$, so $\phi_i\in \End_{H}(W)$. Hence $\End_H(V) \subseteq \End_{H}(W)+i\End_{H}(W) $. The another inclusion is clearly right. Therefore $\End_H(V)\simeq \End_{H}(W)+i\End_{H}(W)$. Similarly,   $\End_{H'}(V) \simeq \End_{H'}(W)+i\End_{H'}(W)$. If $\psi_1+i\psi_2 \in \End_{H}(V)'$, then $\psi_i \circ\phi=\phi\circ \psi_i$, for any $\phi\in \End_{H}(W)$. Then by  \cite[p.13, lemme]{MVW}, $\psi_i \in  \End_{H'}(W)$. Therefore, $\End_{H}(V)'=\End_{H'}(V)$.
\end{proof}

If $H$  is a   complex group, having a    compact real form $K_H$.  It is known that the restriction of  a  finite  dimensional algebraic regular(cf. \cite[Section 1.5]{GoWa})  irreducible representation of $H$  to $K_H$ is also irreducible. In this case, we can replace $H$  by $K_H$, and then  $\End_{H}(V)=\End_{K_H}(V)$.

Let $(H_1, H_2)$ be a pair of compact groups. Assume $(H_1,H_2)=$ one above $(H, H')$ or their  compact real forms. Up to an isomorphism, we modify $\omega|_{H_1H_2}$ to be a unitary representation; under such   a change,  the lemma \ref{doucom} also holds.
\begin{lemma}
$\omega$ is a theta representation of $H_1\times H_2$.
\end{lemma}
\begin{proof}
It follows from Prop.\ref{theta} and Lmm.\ref{doucom}.
\end{proof}
We will apply Thm.\ref{Firstthm} to these representations from isometric groups to some similitude  groups.   Let  us  list some cases by following Vignera's work \cite[Chap.1]{MVW}. Notations: $J=\begin{bmatrix} 0& -I\\
I& 0\end{bmatrix} \in \GL(2n, \C)$,  $F=\R \textrm{ or } \C$, $\T=\{ g\in \C^{\times} \mid |g|=1\}$,
 $$U(n)=\{ g\in \GL(n,\C) \mid \overline{g}^{T}g=I\},$$
 $$O(n, \C)=\{ g\in \GL(n,\C) \mid g^{T}g=I\}, $$
  $$O(n)=\{ g\in \GL(n, \R)\mid g^Tg=I\}, $$
$$\Sp(n, F)=\{ g\in \GL(2n, F)\mid g^T Jg=J\},$$
   $$\Sp(n)=U(2n)\cap \Sp(n, \C) \simeq \{g\in \GL(n, \mathbb{H}) \mid \overline{g}g=I\},$$
 $$\GSp(n, F)=\{ g\in \GL(2n, F)\mid g^T Jg=\lambda(g)J, \lambda(g) \in F^{\times}\}, $$
$$\GU(n,\C)=\{ g\in \GL(n, \C)\mid \overline{g}^{T}g=\lambda(g)I, \lambda(g) \in \R^{\times}\}, $$
$$\GO(n,\C)=\{ g\in \GL(n,\C) \mid g^{T}g=\lambda(g)I, \lambda(g) \in \C^{\times} \},$$
$$TSp(n)=\{ g\in \GSp(n, \C)\mid g=h\diag(t, \cdots, t), h\in \Sp(n), t \in \T\}, $$
$$TO(n)=\{ g\in \GL(n,\C) \mid g=h\diag(t, \cdots, t), h\in O(n), t \in \T\}. $$
Notice: $TSp(n)/Sp(n)\simeq \T/\{\pm 1\} \simeq TO(n)/O(n)$.
\subsubsection{Example 1}
Let $(W=\C^{2n}, \langle, \rangle)$ be a canonical symplectic vector space over $\C$,  $U(W)=\Sp(n, \C)$.\\
II) $W=[X_1\otimes_{\C} X_2] \oplus [X_1^{\ast}\otimes_{\C} X_2^{\ast}]$, $ X_i=\C^{n_i}$, $n=n_1n_2$, $H= \GL(n_1, \C)$, $H'= \GL(n_2, \C)$, $H_1=  U(n_1)$, $H_2=  U(n_2)$,  $G_1=TU(n_1)$, $G_2=TU(n_2)$, $\Gamma=\{ (g_1, g_2)\in G_1\times G_2\mid \lambda(g_1)\lambda(g_2)=1\}$.\\
I)  $W=W_1\otimes_{\C} W_2$, $W_1=\C^{n_1}$ a symmetric vector space over $\C$, $W_2=\C^{2n_2}$ a symplectic vector space over $\C$, $n=n_1n_2$,  $H\simeq  O(n_1,\C)$, $H'\simeq \Sp(n_2,\C)$, $H_1= O(n_1)$, $H_2= Sp(n_2)$, $G_1= TO(n_1)$, $G_2= TSp(n_2)$, $\Gamma=\{ (g_1, g_2)\in G_1\times G_2 \mid g_1g_2\in \Sp(n, \C)\}$.
\subsubsection{Example 2}
Let $(W=\R^{2n}, \langle, \rangle)$ be a canonical symplectic vector space over $\R$,  $U(W)=\Sp(n, \R)$, $V=\C^{2n}$.\\
Ia) $W=[X_1\otimes_{\C} X_2] \oplus [X_1^{\ast}\otimes_{\C} X_2^{\ast}]$, $X_i=\C^{n_i}$,  $n=2n_1n_2$, $H= \GL(n_1, \C)$, $H'= \GL(n_2, \C)$, $H_1=  U(n_1)$, $H_2=  U(n_2)$,  $G_1=TU(n_1)$, $G_2=TU(n_2)$, $\Gamma=\{ (g_1, g_2)\in G_1\times G_2\mid \lambda(g_1)\lambda(g_2)=1\}$. \\
Ib)  $W=W_1\otimes_{\C} W_2$, $ \langle, \rangle=\Tr(\langle, \rangle_1\otimes \langle, \rangle_2)$,  $W_1=\C^{n_1}$  a symmetric vector space over $\C$, $W_2=\C^{2n_2}$ a symplectic vector space over $\C$, $n=2n_1n_2$, $H\simeq O(n_1, \C)$, $H'\simeq \Sp(n_2, \C)$, $H_1=O(n_1)$, $H_2=\Sp(n_2)$, $G_1= TO(n_1)$,$G_2= TSp(n_2)$, $\Gamma=\{ (g_1, g_2)\in G_1\times G_2 \mid g_1g_2\in \Sp(n, \C)\}$.
\subsubsection{Example 3}
Let $(W=\C^{n}, \langle, \rangle)$ be a canonical  symmetric  vector space over $\C$,  $U(W)=O(n,\C)$.\\
Ia)  $W=W_1\otimes_{\C} W_2$, $W_i=\C^{n_i}$ symmetric vector spaces over $\C$, $n=n_1n_2$, $H\simeq O(n_1, \C)$, $H'\simeq O(n_2, \C)$, $H_1=O(n_1)$, $H_2=O(n_2)$, $G_1= TO(n_1)$, $G_2= TO(n_2)$ .\\
Ib) $W=W_1\otimes_{\C} W_2$, $W_i=\C^{2n_i}$ symplectic  vector spaces over $\C$, $n=4n_1n_2$, $H\simeq Sp(n_1, \C)$, $H'\simeq Sp(n_2, \C)$, $H_1=Sp(n_1)$, $H_2=Sp(n_2)$, $G_1= TSp(n_1)$, $G_2= TSp(n_2)$.
\subsubsection{Example 4}
Let $(W=\R^{n}, \langle, \rangle)$ be a  symmetric  vector space over $\R$ of type  $(n,0)$,     $U(W)\simeq O(n)$.\\
 Ia) $W=W_1\otimes_{\R} W_2$,  $W_i$  symmetric vector spaces over $\R$ of type $(n_i,0)$,  $H_i\simeq O(n_i)$, $G_i=TO(n_i)$.
\subsubsection{Example 5}
Let $(W=\C^{n}, \langle, \rangle)$ be an $\epsilon$-Hermitian vector space over $\C$,  $U(W)\simeq U(n)$\\
Ia) $W=W_1\otimes_{\C} W_2$, $W_i=\C^{n_i}$   $\epsilon_i$-Hermitian vector spaces over $\C$, $i=1,2$, $\epsilon_1\epsilon_2=\epsilon$, $n=n_1n_2$, $H_i=U(n_i)$.

\labelwidth=4em
\addtolength\leftskip{25pt}
\setlength\labelsep{0pt}
%\addtolength\parskip{\smallskipamount}


\begin{thebibliography}{99}

 \bibitem{AdBa} J. Adams, D. Barbasch,
 {\it Reductive dual pair correspondence for complex groups}, J. Funct. Anal. 132, no. 1 (1995), 1-42.

\bibitem{Ad1} J. Adams,
 {\it The Theta Correspondence Over $\R$,  Harmonic Analysis, Group Representations, Automorphic Forms and Invariant Theory}, 1-39. Lect. Notes Ser. Inst. Math. Sci. Natl. Univ. Singap. 12. 2007.


\bibitem{AAG}A. Aizenbud, N. Avni, and D. Gourevitch,
 {\it Spherical pairs over close local fields}, Comment. Math. Helv. 87 (2012), 929-962.

\bibitem{AM} T.Austin, C. C. Moore,
{\it Continuity properties of measurable group cohomology}, Math. Ann. 356 (2013), 885-937.

\bibitem{BeHa} B. Bekka, P. de la Harpe,
 {\it Unitary representations of groups, duals, and characters}, an expository book, arXiv:1912.07262.

\bibitem{BernZ} I.N. Bernstein, A.V.Zelveinsky,
 {\it Representations of the group $\GL(n,F)$ where $F$ is a non-archimedean local field}, Russ. Math. Surv. 31(3), 1-68(1976).


\bibitem{Bo} N. Bourbaki,
 {\it Espaces vectoriels topologiques}, Chap. 1 \`a 5, Springer-Verlag, Berlin, 2007.


\bibitem{Br} F. Bruhat,
 {\it Lectures on Lie Groups and Representations of Locally Compact Groups}, Tata Institute of Fundamental Research, Bombay, 1968.


\bibitem{BushH} C.J. Bushnell, G. Henniart,
 {\it The Local Langlands Conjecture for $GL(2)$},
 Grundlehren der Mathematischen Wissenschaften, vol. 335, Springer-Verlag, Berlin, 2006.

\bibitem{Ca2} W. Casselman,
{\it Canonical extensions of Harish-Chandra modules to representations of $G$}, Canad. J. Math. 41 (3) (1989) 385-438.

\bibitem{Ca} W. Casselman,
{\it Introduction to the theory of admissible representations of p -adic reductive groups},
preprint, 1995.



\bibitem{Co} D.L.Cohn,
{\it Measure Theory}, a Birkh\"auser Advanced Texts: Basler Lehrb\"ucher, Birkh\"auser/ Springer, New York, 1980.(Reprint in China)


\bibitem{Co1} J.Conway,
{\it A course in functional analysis}, Graduate Texts in Mathematics, vol. 96, 2nd edn., Springer, New York, 1990.

\bibitem{Co2} J.Conway,
{\it A Course in Operator Theory}, Amer. Math. Soc., New York, 2000.

\bibitem{CH}Y. de Cornulier, P. de la Harpe,
{\it Metric Geometry of Locally Compact Groups},
 EMS Tracts in Mathematics, vol. 25, European Mathematical Society (EMS), Z\"urich, 2016.

\bibitem{CuRe} C.W. Curtis, I. Reiner,
{\it Methods of Representation Theory. Vol. I. With
applications to finite groups and orders},
Wiley Classics Lib. John Wiley and Sons, 1990.


\bibitem{Fa} R.C. Fabec,
{\it Fundamentals of Infinite Dimensional Representation Theory},
 Chapman \& Hall/CRC,  2000.


\bibitem{Fo} G.B. Folland,
{\it A course in abstract harmonic analysis}, Studies in Advanced Mathematics, CRC Press, Boca Raton, FL, 1995.



 \bibitem{Ga} W.T.Gan,
{\it Theta correspondence: recent progress and applications}, Proceedings of the International Congress of Mathematicians-Seoul 2014. Vol. II, 343-366, Kyung Moon Sa, Seoul, 2014.


 \bibitem{Gan} W.T.Gan, S. Takeda.
{\it A proof of the Howe duality conjecture}, J. Amer. Math. Soc. 29, no. 2 (2016), 473-493.

\bibitem{Ga2} P. Garrett, Notes at  http://www.math.umn.edu/egarrett/.

\bibitem{Ge} S. Gelbart, {\it Weil's representation and the spectrum of the metaplectic group}, Lecture Notes in Math., vol 530, Springer-Verlag, Berlin, 1976.

\bibitem{GoWa} R. Goodman, N. R. Wallach, {\it Symmetry, representations, and invariants}, Graduate Texts in Mathematics, vol. 255, Springer, 2009.

\bibitem{HM1} E. Hewitt, K. Ross,
{\it Abstract Harmonic Analysis I. Structure of Topological groups, Integration Theory, Group Representations}, Second Edition, Grundlehren der Mathematischen Wissenschaften, vol. 115, Springer-Verlag, Berlin,  1979.(Reprint in China)

\bibitem{HM2} E. Hewitt, K. Ross,
{\it Abstract harmonic analysis II. Structure and analysis for compact groups, Analysis on locally compact Abelian groups},  Grundelehern der Mathematischen Wissenschasften, vol. 152, Springer-Verlag,  Berlin,  1970.(Reprint in China)



\bibitem{HM} K.H. Hofmann, S.A. Morris,
{\it The Structure of Compact Groups}, third edition, revised and augmented, De Gruyter Studies in Mathematics, vol. 25, De Gruyter, Berlin, 2013.

\bibitem{Ho1} R. Howe,
{\it  Remarks on classical invariant theory}, Trans. Amer. Math. Soc. 313(2)(1989), 539-570.

\bibitem{Ho2} R. Howe,
{\it Transcending classical invariant theory}, J. Amer. Math. Soc. 2 (1989), 535-552.

\bibitem{Ho3} R. Howe,  et al., {\it  Lectures: In honor of Professor Roger Howe on the occasion of his 70th birthday}, 2015. https://math.mit.edu/conferences/howe/lectures.php.

\bibitem{Is} I. M. Isaacs,
{\it Character Theory of Finite Groups}, AMS Chelsea Publishing, Providence, RI, 2006.


\bibitem{KT} E. Kaniuth, K.F. Taylor,
{\it Induced Representations of Locally Compact Groups}, Cambridge Tracts in Math., vol. 197, Cambridge University Press, Cambridge, 2013.


\bibitem{K2} A. Kerber,
{\it  Representations of Permutation Groups II},  Lecture Notes in Math., vol. 495,  Springer-Verlag,  1975.


\bibitem{KL} A. Kleppner, R. L. Lipsman,
{\it The Plancherel formula for group extensions. I},  Ann. Sci. \'Ecole Norm. Sup.  (1972), 459-516.


\bibitem{Ku} S.Kudla,
{\it On the local theta-correspondence}, Invent. Math. 83, no. 2 (1986), 229-255.

\bibitem{Ku2} S.Kudla,
 {\it  Notes on the local theta correspondence (lectures at the European School in Group Theory)}, preprint, available at http://www.math.utotonto.ca/~skudla/castle.pdf, 1996.

\bibitem{Ma1} G.W.Mackey,
{\it Induced representations of locally compact groups. I}, Ann. of Math. (2) 55 (1952), 101-139.


\bibitem{Ma2} G.W.Mackey,
{\it  Borel structure in groups and their duals}, Trans. Am. Math. Soc. 85(1957), 134-165.


\bibitem{Ma3}G.W.Mackey,
{\it Unitary representations of group extensions I}, Acta Math. 99 (1958), 265-311.

\bibitem{Ma4} G.W.Mackey,
{\it The Theory of Unitary Group Representations}, University of Chicago Press, Chicago, Ill.-London, 1976.



\bibitem{MVW} C.Moeglin,  M.-F.Vigneras, J.-L. Waldspurger,
{\it Correspondances de Howe sur un corps $p$-adique},
Lecture Notes in Math.,  vol. 1921, Springer-Verlag, New York, 1987.


\bibitem{Mo0} C.C. Moore,
{\it  On the Frobenius reciprocity theorem for locally compact groups}, Pacific J. Math.12 (1) (1962), 359-365.

\bibitem{Mo1} C.C. Moore,
{\it Extensions and low dimensional cohomology theory of locally compact groups. I, II}, Trans. Amer. Math. Soc. 113 (1964).

\bibitem{Mo2} C.C. Moore,
{\it Group extensions of p-adic and adelic linear groups}, Publ. Math. IH\'ES 35 (1968), 157-222.

\bibitem{Mo3} C.C. Moore,
{\it Group extensions and cohomology for locally compact groups. III, IV}, Trans. Amer. Math. Soc. 221 (1976),  1-33, 35-58.

 \bibitem{Pr} D.Prasad,
{\it Ext-analogues of branching laws}, Proceedings of the International Congress of Mathematicians-Rio de Janeiro 2018. Vol. II. Invited lectures, 1367-1392, World Sci. Publ., Hackensack, NJ, 2018.

 \bibitem{Ra}  S. Rallis,
{\it On the Howe duality conjecture}, Compos. Math. 51 (1984), 333-399.

\bibitem{Ri} M.A. Rieffel,
{\it Normal subrings and induced representations}, J. Algebra 59 (1979), 364-386.

\bibitem{Ru} H. Rubenthaler,
{\it Les paires duales dans les algebr\`es de Lie r\'eductives}, Ast\'erisque 219 (1994).

\bibitem{Wa} J.-L.Waldspurger,
{\it D\'emonstration d'une conjecture de dualitšŠ de Howe dans le cas $p$-adique,$p\neq2$}, in Festschrift in Honor of I. I. Piatetski-Shapiro on the Occasion of his Sixtieth Birthday, Israel Math. Conf. Proc. 2, Weizmann, Jerusalem, 1990, pp. 267-324.


\bibitem{Wan} C.-H. Wang,
{\it On the local theta representation}, preprint.

\bibitem{War} G. Warner,
{\it Harmonic Analysis on Semi-Simple Lie Groups. I}, Springer-Verlag, New York, 1972.

\bibitem{We} A. Weil,
{\it L'int\'egration dans les groupes topologiques et ses applications}, Actualiti\'es Scientifiques et Industrielles, no.869, Paris, 1938.

\bibitem{We2} A. Weil,
{\it  Sur certains groupes d'op\'erateurs unitaires},
Acta Math. 111 (1964), 143-211.

\end{thebibliography}
\end{document}